\documentclass[a4paper,twoside,10pt]{article}

\usepackage{amsmath, mathrsfs,amssymb}
\usepackage[dvipsnames]{xcolor}
\numberwithin{equation}{section}

\usepackage{mathtools}

\usepackage{scalerel}
\usepackage[usestackEOL]{stackengine}
\def\dashint{\,\ThisStyle{\ensurestackMath{%
            \stackinset{c}{.2\LMpt}{c}{.5\LMpt}{\SavedStyle-}{\SavedStyle\phantom{\int}}}%
        \setbox0=\hbox{$\SavedStyle\int\,$}\kern-\wd0}\int}

\usepackage{amsthm}
\usepackage{aliascnt}

\newtheorem{theorem}{Theorem}[section]

\newaliascnt{lemma}{theorem}
\newtheorem{lemma}[lemma]{Lemma}
\aliascntresetthe{lemma}

\newaliascnt{proposition}{theorem}
\newtheorem{proposition}[proposition]{Proposition}
\aliascntresetthe{proposition}

\newaliascnt{corollary}{theorem}
\newtheorem{corollary}[corollary]{Corollary}
\aliascntresetthe{corollary}

\theoremstyle{definition}
\newaliascnt{definition}{theorem}

\aliascntresetthe{definition}

\theoremstyle{remark}
\newaliascnt{remark}{theorem}
\newtheorem{remark}[remark]{Remark}
\aliascntresetthe{remark}

\usepackage{dsfont}
\newcommand{\N}{\mathds N}
\newcommand{\R}{\mathds R}
\newcommand{\weak}{\rightharpoonup}
\newcommand{\eps}{\varepsilon}

\usepackage[a4paper,centering,margin=2cm,bindingoffset=0cm,marginpar=0cm]{geometry} 

\usepackage{fancyhdr}
\usepackage{titlesec}
\titleformat{\section}{\filcenter\sc\large}{\thesection.\;}{0em}{}
\titleformat{\subsection}[runin]{\bf}{\thesubsection.\;}{0em}{}[.]

\usepackage[backend=biber,maxnames=10,firstinits=true]{biblatex}
\DeclareFieldFormat[article,inbook,incollection,inproceedings,patent,thesis,unpublished]{title}{{\em#1}}
\DeclareFieldFormat[article,inbook,incollection,inproceedings,patent,thesis,unpublished]{journaltitle}{#1}
\DeclareFieldFormat[incollection,inproceedings]{booktitle}{#1}
\AtEveryBibitem{%
  \clearname{translator}%
  \clearfield{pagetotal}%
  \clearlist{language}%
  \clearfield{month}
}
\renewbibmacro{in:}{}
\DeclareFieldFormat{eprint:euclid}{Project euclid: \href{http://projecteuclid.org/euclid.\thefield{eprint}}{\thefield{eprint}}}

\usepackage[pdftex,colorlinks=true,linkcolor=blue,citecolor=red,bookmarks=true,bookmarksnumbered=true]{hyperref}
\hypersetup{%
	pdfauthor={Franziska Borer, Peter, Elbau, Tobias Weth}, 
	pdfsubject={Prescribed Curvature Flow on Closed Surfaces with Negative Euler Characteristic}, 
	pdftitle={Prescribed Curvature Flow on Surfaces with Negative Euler Characteristic}, 
	pdfkeywords={Prescribed Curvature Flow, global solution, convergence, evolution equation}}

\newcommand{\e}{\mathrm{e}}

\DeclareMathOperator{\vol}{vol}

\title{A Variant Prescribed Curvature Flow on Closed Surfaces with Negative Euler Characteristic}
\author{Franziska Borer\thanks{Technical University of Berlin, Institute of Mathematics, Stra{\ss}e des 17. Juni 136, D-10623 Berlin, Germany\newline email: \href{mailto:borer@tu-berlin.de}{borer@tu-berlin.de}}
\and Peter Elbau\thanks{University of Vienna, Faculty of Mathematics, Oskar-Morgenstern-Platz 1, A-1090 Vienna, Austria\newline email: \href{mailto:peter.elbau@univie.ac.at}{peter.elbau@univie.ac.at}}
\and Tobias Weth\thanks{Goethe University Frankfurt, Institute of Mathematics, Robert-Mayer-Stra{\ss}e 10, D-60629 Frankfurt, Germany\newline email: \href{mailto:weth@math.uni-frankfurt.de}{weth@math.uni-frankfurt.de}}}
\date{}

\begin{document}
\maketitle

\begin{abstract}
  On a closed Riemannian surface $(M,\bar g)$ with negative Euler characteristic, we study the problem of finding conformal metrics  with prescribed volume $A>0$ and the property that their Gauss curvatures $f_\lambda= f + \lambda$ are given as the sum of a prescribed function $f \in C^\infty(M)$ and an additive constant $\lambda$. Our main tool in this study is a new variant of the prescribed Gauss curvature flow, for which we establish local well-posedness and global compactness results. In contrast to previous work, our approach does not require any sign conditions on $f$. Moreover, we exhibit conditions under which the function $f_\lambda$ is sign changing and the standard prescribed Gauss curvature flow is not applicable. 
\end{abstract}

\section*{Acknowledgment}
This work was funded by the Deutsche Forschungsgemeinschaft (DFG, German Research
Foundation), project 408275461 (Smoothing and Non-Smoothing via Ricci Flow).\\
We would like to thank Esther Cabezas--Rivas for helpful discussions.


\section{Introduction}

Let $(M,\bar g)$ be a two-dimensional, smooth, closed, connected, oriented 
Riemann manifold endowed with a smooth background metric $\bar g$. A classical problem raised by Kazdan and Warner in \cite{KazWar74_1} and \cite{KazWar74_2} is the question 
which smooth functions $f \colon M \to \R$ arise as the Gauss 
curvature $K_g$ of a conformal metric $g(x)=\e^{2u(x)}\bar g(x)$ on $M$ and to characterise the set of all
such metrics.

For a constant function $f$, this prescribed Gauss curvature problem is exactly the statement of the {\em Uniformisation Theorem (see e.g.~\cite{Poi08}, \cite{Koe08}):}\\
{\em There exists a metric $g$ which is pointwise conformal to $\bar g$ and
has constant Gauss curvature $K_{g}\equiv\bar K\in\R$.}\\
We now use this statement to
assume in the following without loss of generality
that the background metric $\bar g$ itself has constant Gauss curvature $K_{\bar g}\equiv \bar K\in\R$. Furthermore we can normalise the volume of $(M,\bar g)$ to one. We recall that the Gauss curvature of a conformal 
metric $g(x)=\e^{2u(x)}\bar g(x)$ on $M$ is given by the Gauss equation
\begin{equation}\label{GaussEquation}
K_{g}(x) = \e^{-2u(x)}(-\Delta_{\bar g}u(x) + \bar K).
\end{equation}
Therefore the problem reduces to the question for which functions $f$ there exists a conformal factor $u$ solving the equation 
\begin{equation}\label{PCP}
-\Delta_{\bar g}u(x) + \bar K = f(x) \e^{2u(x)} \qquad \text{in $M$.}
\end{equation}
Given a solution $u$, we may integrate \eqref{PCP} with respect to the measure $\mu_{\bar g}$ on $M$ induced by the Riemannian volume form. Using the 
Gauss--Bonnet Theorem, we then obtain the identity 
\begin{equation}\label{Condition1}
\int_M f(x) d\mu_g(x) = \int_M \bar K d\mu_{\bar g}(x) = \bar K \vol_{\bar g}=\bar K= 2\pi \chi(M),
\end{equation}
where $d\mu_g(x) = \e^{2u(x)}d\mu_{\bar g}(x)$ is the element of area in the metric $g(x)=\e^{2u(x)}\bar g(x)$.
We note that \eqref{Condition1} immediately yields necessary conditions on $f$ for the solvability of the prescribed Gauss curvature problem. In particular, if $\pm \chi(M)>0$, then $\pm f$ must be positive somewhere. Moreover, if $\chi(M)=0$, then $f$ must change sign or must be identically zero.

In the present paper we focus on the case $\chi(M)<0$, so $M$ is a surface of genus greater than one and $\bar K < 0$. The complementary cases $\chi(M) \ge 0$---i.e., the cases where $M = S^2$ or $M=T$, the $2$-torus---will be discussed briefly at the end of this introduction, and we also refer the reader to \cite{Str05,Str20,BuzSchStr16,Gal15} and the references therein. Multiplying equation \eqref{PCP} with the factor $\e^{-2u}$ and integrating over $M$ with respect to the measure $\mu_{\bar g}$,
we get the following necessary condition---already mentioned by Kazdan and Warner in \cite{KazWar74_1}---for the average $\bar f:=\frac{1}{\vol_{\bar g}}\int_M f(x)d\mu_{\bar g}(x)$, with $\vol_{\bar g}:=\int_Md\mu_{\bar g}(x)$:
\begin{equation}\label{Condition2}
\begin{split}
\bar f&=\frac{1}{\vol_{\bar g}}\int_M f(x)d\mu_{\bar g}(x)=\int_M(-\Delta_{\bar g}u(x)+\bar K)\e^{-2u(x)}d\mu_{\bar g}(x)\\
&=\int_M(-2|\nabla_{\bar g}u(x)|^2_{\bar g}+\bar K)\e^{-2u(x)}d\mu_{\bar g}(x)<0.
\end{split}
\end{equation}
This condition is not sufficient. Indeed, it has already been pointed out in \cite[Theorem 10.5]{KazWar74_1} that in the case $\chi(M)<0$ there always exist functions $f \in C^\infty(M)$ with $\bar f< 0$ and the property that \eqref{PCP} has no solution.

We recall that solutions of \eqref{PCP} can be 
characterised as critical points of the functional 
\begin{equation}\label{Ef}
E_f:H^1(M,\bar g)\to \R;\quad E_f(u) := \frac12\int_M \left(|\nabla_{\bar g} u(x)|^2_{\bar g} + 2\bar K u(x) - f(x)\e^{2u(x)}\right)
 d\mu_{\bar g}(x).   
\end{equation}
Under the assumption $\chi(M)<0$, i.e., $\bar K < 0$, the functional $E_f$ is strictly convex and coercive on $H^1(M,\bar g)$ if $f\le 0$ and $f$ 
does not vanish identically. Hence, as noted in \cite{DinLiu95}, the functional $E_f$ admits a unique critical
point $u_f \in H^1(M,\bar g)$ in this case, which is a strict absolute minimiser of $E_f$ and a (weak) solution of \eqref{PCP}. The situation is more delicate in the case where $f_\lambda=f_0+\lambda$, where $f_0\le0$ is a smooth, nonconstant function on $M$ with $\max_{x\in M}f_0(x)=0$, and $\lambda >0$. In the case where $\lambda>0$ sufficiently small (depending on $f_0$), it was shown in \cite{DinLiu95} and \cite{BorGalStr15} that the corresponding functional $E_{f_\lambda}$ admits a local minimiser $u_\lambda$ and a further critical point $u^\lambda\neq u_\lambda$ of mountain pass type.

These results motivate our present work, where we suggest a new flow approach to the prescribed Gausss curvature problem in the case $\chi(M)<0$. It is important to note here that there is an intrinsic motivation to formulate the static problem in a flow context. 
Typically, elliptic theories are regarded as the static case of the corresponding parabolic problem; in that sense, many times the better-understood elliptic theory
has been a source of intuition to generalise the corresponding results in the
parabolic case. Examples of this feedback are minimal surfaces/mean curvature
flow, harmonic maps/solutions of the heat equation, and the Uniformisation Theorem/the two-dimensional normalised Ricci flow.

In this spirit, a flow approach to \eqref{PCP}, the so-called prescribed Gauss curvature flow, was first introduced by Struwe in \cite{Str05} (and \cite{BuzSchStr16}) for the case $M=S^2$ with the standard background metric and a positive function $f \in C^2(M)$. More precisely, he considers a family of metrics $(g(t,\cdot))_{t\ge0}$ which fulfils the initial value problem
\begin{align}
\partial_tg(t,x)&=2(\alpha(t)f(x)-K_{g(t,\cdot)}(x))g(t,x)\quad\text{in }(0,T)\times M;\\
g(0,x)&= g_0(x)\quad\text{on }\{0\}\times M,
\end{align}
with 
\begin{equation}
  \label{eq:def-alpha-t}
\alpha(t)=\frac{\int_MK_{g(t,\cdot)}(x)d\mu_{g(t,\cdot)}(x)}{\int_M f(x)d\mu_{g(t,\cdot)}(x)}=\frac{2\pi\chi(M)}{\int_Mf(x)d\mu_{g(t,\cdot)}(x)}.
\end{equation}
This choice of $\alpha(t)$ ensures that the volume of $(M,g(t,\cdot))$ remains constant throughout the deformation, i.e.,
\[\int_M d\mu_{g(t,\cdot)}(x)=\int_M\e^{2u(t,x)}d\mu_{\bar g}(x)\equiv\vol_{g_0}\quad\text{for all }t\ge0,\]
where $g_0$ denotes the initial metric on $M$. Equivalently one may consider the evolution equation for the associated conformal factor $u$ given by $g(t,x)=\e^{2u(t,x)}\bar g(x)$: 
\begin{align}
\partial_tu(t,x)&=\alpha(t) f(x)-K_{g(t,\cdot)}(x)\quad\text{in }(0,T)\times M;\label{PCFnegative}\\
u(0,x)&=u_0(x)\quad\text{on }\{0\}\times M.
\end{align}
Here the initial value $u_0$ is given by $g_0(x)=\e^{2u_0(x)}\bar g(x)$. The flow associated to this parabolic equation is usually called the prescribed Gauss curvature flow. 
 With the help of this flow, Struwe \cite{Str05} provided a new proof of a result by Chang and Yang \cite{ChaYan87} on sufficient criteria for a function $f$ to be the Gauss curvature of a metric $g(x)=\e^{2u(x)}g_{S^2}(x)$ on $S^2$. He also proved the sharpness of these criteria.
 
In the case of surfaces with genus greater than one, i.e., with negative Euler characteristic, the prescribed Gauss curvature flow was used by Ho in \cite{Ho11} to prove that any smooth, strictly negative function on a surface with negative Euler characteristic can be realised as the Gaussian curvature of some metric.
More precisely, assuming that $\chi(M)< 0$ and that $f \in C^\infty(M)$ is a strictly negative function, he proves that 
equation \eqref{PCFnegative} has a solution which is defined for all times and converges to a metric $g_\infty$ with Gaussian curvature $K_{g_\infty}$ satisfying
\[K_{g_\infty}(x)=\alpha_\infty f(x)\]
for some constant $\alpha_\infty$.

While the prescribed Gauss curvature flow is a higly useful tool in the cases where $f$ is of fixed sign,   
it cannot be used in the case where $f$ is sign-changing. Indeed, in this case we may have $\int_Mf(x)d\mu_{g(t,\cdot)}(x)=0$ along the flow and then the normalising factor $\alpha(t)$ is not well-defined by \eqref{eq:def-alpha-t}.
As a consequence, a long-time solution of \eqref{PCFnegative} might not exist. In particular, the static existence results of \cite{DinLiu95} and \cite{BorGalStr15} can not be recovered and reinterpreted with the standard prescribed Gauss curvature flow.

In this paper we develop a new flow approach to \eqref{PCP} in the case $\chi(M)<0$ for general $f \in C^\infty(M)$, which sheds new light on the results in \cite{DinLiu95}, \cite{BorGalStr15} and \cite{Ho11}.
The main idea is to replace the multiplicative normalisation in \eqref{PCFnegative} by an additive normalisation, as will be described in details in the next chapter.

At this point, it should be noted that the normalisation factor $\alpha(t)$ in the prescribed Gauss curvature flow given by \eqref{eq:def-alpha-t} is also not the appropriate choice in the case of the torus, where, as noted before, $f$ has to change sign or be identically zero in order to arise as the Gauss curvature of a conformal metric. The case of the torus was considered by  Struwe in \cite{Str20}, where, in particular, he used to a flow approach to reprove and partially improve a result by Galimberti \cite{Gal15} on the static problem. In this approach, the normalisation in \eqref{eq:def-alpha-t} is replaced by 
\begin{equation}
  \label{eq:def-alpha-t-torus}
\alpha(t)=\frac{\int_M f(x)K_{g(t,\cdot)}(x)d\mu_{g(t,\cdot)}(x)}{\int_Mf^2(x)d\mu_{g(t,\cdot)}(x)}.
\end{equation}
With this choice, Struwe shows that  for any smooth
$$
u_0\in C^*:=\left\{u\in H^1(M,\bar g)\mid\int_Mf(x)\e^{2u(x)}d\mu_{\bar g}(x)=0,\:\int_M\e^{2u(x)}d\mu_{\bar g}(x)=1\right\}
$$
there exists a unique, global smooth solution $u$ of \eqref{PCFnegative} satisfying $u(t,\cdot)\in C^*$ for all $t>0$. Moreover, $u(t,\cdot)\to u_\infty(\cdot)$ in $H^2(M,\bar g)$ (and
smoothly) as $t\to\infty$ suitably, where $u_\infty+c_\infty$ is a smooth solution of \eqref{PCP} for some $c_\infty\in\R$.

In principle, the normalisation \eqref{eq:def-alpha-t-torus} could also be considered in the case $\chi(M)<0$, but then the flow is not volume-preserving anymore, which results in a failure of uniform estimates for solutions of \eqref{PCFnegative}. Consequently, we were not able to make use of the associated flow in this case. 

The paper is organised as follows. In \autoref{SectionProperties} we set up the framework for the new variant of the prescribed Gauss curvature flow with additive normalisation, and we collect basic properties of it. In \autoref{SectionMainResults}, we then present our main result on the long-time existence and convergence of the flow (for suitable times $t_k \to \infty$) to solutions of the corresponding static problem. In particular, our results show how sign changing functions of the form $f_\lambda = f_0 + \lambda$ arise depending on various assumptions on the shape of $f_0$ and on the fixed volume $A$ of $M$ with respect to the metric $g(t)$. Before proving our results on the time-dependent problem, we first derive, in \autoref{sec:stat-minim-probl}, some results on the static problem with volume constraint. Most of these results will then be used in \autoref{sec:proof-main-results}, where the parabolic problem is studied in detail and the main results of the paper are proved. In the appendix, we provide some regularity estimates and a variant of a maximum principle for a class of linear evolution problems with H\"older continuous coefficients.   

In the remainder of the paper, we will use the short form $f$, $g(t)$, $u(t)$, $K_{g(t)}$, $\vol_{g(t)}:=\int_M d\mu_{g(t)}=\int_M\e^{2u(t)}d\mu_{\bar g}$, and so on instead of $f(x)$, $g(t,x)$, $u(t,x)$, $K_{g(t,\cdot)}(x)$, $\int_M d\mu_{g(t,\cdot)}(x)=\int_M\e^{2u(t,x)}d\mu_{\bar g}(x)$, et cetera.

\section{A New Flow Approach and Some of its Properties}\label{SectionProperties}
Before introducing the additively rescaled prescribed Gauss curvature flow, we recall an important and highly useful estimate. The following lemma (see e.g.~\cite[Corollary 1.7]{Cha04}) is a consequence of the Trudinger's inequality \cite{Tru67} which was improved by Moser in \cite{Mos71} (for more details see e.g.~\cite[Theorem 2.1 and Theorem 2.2]{Str20}):

\begin{lemma}\label{Onofri-ungleichung}
For a two-dimensional, closed Riemannian  manifold $(M,\bar g)$ there are constants $\eta>0$ and $C_{\text{MT}}>0$ such that
\begin{equation}\label{OU}
\int_M\e^{(u-\bar u)}d\mu_{\bar g}\le C_{\text{MT}}\exp\left(\eta\|\nabla_{\bar g}u\|^2_{L^2(M,\bar g)}\right)
\end{equation}
for all $u\in H^1(M,\bar g)$ where
\[\bar u:=\frac{1}{\vol_{\bar g}}\int_Mu\:d\mu_{\bar g}=\int_Mu\:d\mu_{\bar g},\]
in view of our assumption that $\vol_{\bar g}=1$.
\end{lemma}

As a consequence of \autoref{Onofri-ungleichung}, we have
\begin{equation}
  \label{eq:onofri-consequence}
\int_M \e^{p u}d\mu_{\bar g} = \e^{p \bar u} \int_M\e^{(pu-\bar{pu})}d\mu_{\bar g} \le \e^{p \bar u}
C_{\text{MT}}\exp\left(\eta\|\nabla_{\bar g}(pu)\|^2_{L^2(M,\bar g)}\right)< \infty 
\end{equation}
for every $u \in H^1(M,\bar g)$ and $p>0$. Therefore,  for a given $A>0$, the set 
\begin{equation}\label{NB}
\mathcal{C}_{A}:= \left\{u\in H^1(M,\bar g)\mid \int_M \e^{2u}d\mu_{\bar g}=A\right\}
\end{equation}
is well defined. We also note that  
\begin{equation}
\label{Jensen-consequence}
\bar u\le\frac12\log(A) \qquad \text{for $u\in \mathcal{C}_{A}$,}
\end{equation}
since by Jensen's inequality and our assumption that $\vol_{\bar g}=1$ we have 
\begin{equation}
\label{jensen-consequence}  
2\bar u=\dashint_M 2 ud\mu_{\bar g}=\int_M 2ud\mu_{\bar g}\le \log\left(\dashint\e^{2u}d\mu_{\bar g}\right)=\log(A)\qquad \text{for $u\in \mathcal{C}_{A}$.}
\end{equation}
Next, we let $f\in C^\infty(M)$ be a fixed smooth function. As a consequence of \eqref{eq:onofri-consequence}, the energy functional $E_f$ given in \eqref{Ef} is then well defined and of class $C^1$ on $H^1(M,\bar g)$. Moreover, we have
\begin{equation}
  \label{eq:C-A-energy-ineq}
  E_f(u) \le \frac12 \|\nabla u\|_{L^2(M,\bar g)}^2 + |\bar K| \|u\|_{L^1(M,\bar g)}  + \frac{A}{2} \|f\|_{L^\infty(M,\bar g)} \quad \text{for $u
    \in \mathcal{C}_{A}$}
\end{equation}

We now consider the additively rescaled prescribed Gauss curvature flow given by the evolution equation
\begin{equation}\label{PCFVol1}
  \partial_tu(t)=f-K_{g(t)}-\alpha(t) = f+\e^{-2u(t)}( \Delta_{\bar g}u(t)-\bar K)-\alpha(t) \quad\text{in }(0,T)\times M,
\end{equation}
where $\alpha(t)$ is chosen such that the volume $\vol_{g(t)}$ of $M$ with respect to the metric $g(t)=\e^{2u(t)}\bar g$ remains constant along the flow. The latter condition requires that  
\begin{equation}\label{VolPres}
\frac12\frac{d}{dt}\vol_{g(t)}= \int_M\partial_t u(t) d\mu_{g(t)}=\int_M(f-K_{g(t)}-\alpha(t))d\mu_{g(t)}=\int_Mf d\mu_{g(t)} -\alpha(t)\vol_{g(t)} - \bar K
\end{equation}
vanishes for $t >0$ and therefore suggest the definition of $\alpha(t)$ given in \eqref{eq:definition-alpha} below. We first note the following observations.
\begin{proposition}\label{Properties1}
Let $T>0$, $f \in C^\infty(M)$, $A>0$, let $u_0 \in \mathcal{C}_{A}$, and let $u \in C([0,T), H^1(M,\bar g))\cap C^1((0,T),H^2(M,\bar g))$ be a solution of the initial value problem
\begin{align}
\partial_tu(t)&=f-K_{g(t)}-\alpha(t)\quad\text{in }(0,T)\times M; \label{PCFVolNew1}\\
u(0)&=u_0 \quad\text{on }\{0\}\times M,\label{PCFVolNew2}
\end{align} 
where
\begin{equation}
  \label{eq:definition-alpha}
\alpha(t)=\frac{1}{A}\left(\int_Mf d \mu_{g(t)}-\bar K\right)=\frac{1}{A}\left(\int_Mf \e^{2u(t)}d \mu_{\bar g}-\bar K\right)
\end{equation}
Then  
\begin{enumerate}
\item the volume $\vol_{g(t)}$ of $(M,g(t))$ is preserved along the flow, i.e., $\vol_{g(t)}\equiv \vol_{g_0}=A$ and therefore $u(t) \in \mathcal{C}_{A}$ for $t \in [0,T)$;
\item along this trajectory, we have a uniform bound for $\alpha$ given by
\begin{equation}\label{Lower-Upper-Boundtildealpha}
|\alpha(t)| \le \alpha_0\quad \text{for $t \in [0,T)\;$ with}\quad  \alpha_0:=\|f\|_{L^\infty(M,\bar g)} +\frac{|\bar K|}{A};
\end{equation}
\item the equation \eqref{PCFVolNew1} remains invariant under adding a constant $c \in \R$ to the function $f$;
\item the function $t \mapsto E_f(u(t))$ is decreasing on $[0,T)$, so in particular $E_f(u(t))\le E_f(u_0)$ for $t \in [0,T)$;
\item there exist constants $c_0=c_0(u_0)>0$, $c_1=c_1(u_0)>0$ depending only on $u_0$ with the property that 
\begin{equation}\label{Unigradbound}
\|\nabla_{\bar g}u(t)\|^2_{L^2(M,\bar g)}\le c_0 + c_1 \|f\|_{L^\infty(M,\bar g)} \qquad \text{for $t \in [0,T)$;}
\end{equation}
\item there exist constants $m_0=m_0(u_0) \in \R$, $m_1=m_1(u_0)>0$ depending only on $u_0$ with the property that 
\begin{equation}\label{UniBarU}
m_0- m_1 \|f\|_{L^\infty(M,\bar g)} \le \bar u(t)\le\frac12\log(A) \qquad \text{for $t \in [0,T)$;}
\end{equation}
\item for every $p \in \R$ there exist constants $\nu_0= \nu_0(u_0,p),\,\nu_1 = \nu_1(u_0,p)>0$ with 
\begin{equation}\label{Uniformexp}
\int_M \e^{2pu(t)}d\mu_{\bar g}\le \nu_0 \e^{\nu_1 \|f\|_{L^\infty(M,\bar g)}} \qquad \text{for $t \in [0,T)$.}
\end{equation}
\end{enumerate}
\end{proposition}

\begin{proof}
  1. Let $h(t)=\frac{1}{2}\bigl(\vol_{g(t)}-A\bigr)$. Then by \eqref{VolPres} we have  
  \begin{align*}
    \dot h(t) &=  \frac12\frac{d}{dt}\vol_{g(t)} = \int_Mf d\mu_{g(t)} -\alpha(t)\vol_{g(t)} - \bar K 
                =  \left(\int_Mf d\mu_{g(t)} - \bar K\right)\left(1- \frac{\vol_{g(t)}}{A}\right)\\
              &= \frac{2}{A}\left(\int_Mf d\mu_{g(t)} - \bar K\right) h(t) \qquad \text{for $t \in (0,T).$}
  \end{align*}
  Since $h$ is continuous in $0$ and $h(0)=0$, Gronwall's inequality (see e.g.~\cite{CazHar99}) implies that $h(t)=0$ and therefore $\vol_{g(t)} = A$ for $t \in [0,T)$.\\
2. follows directly from \eqref{eq:definition-alpha}.\\
To show 3., we note that replacing $f$ by $f+ c$ in \eqref{PCFVolNew1} gives
\[f+ c-K_{g(t)}-\frac1A\left(\int_M(f+ c)d\mu_{g(t)}-\bar K\right)=f-K_{g(t)}-\frac1A\left(\int_Mfd\mu_{g(t)}-\bar K\right)=\partial_tu(t),\]
so the equation remains unchanged.\\
To see 4., we use \eqref{VolPres} and get
\begin{equation}\label{EnergyEstimate}
\begin{split}
\frac{d}{dt}E_{f}(u(t))&=\int_M(-\Delta_{\bar g}u(t)+\bar K-f\e^{2u(t)})\partial_tu(t)d\mu_{\bar g}\\
&=\int_M((-\Delta_{\bar g}u(t)+\bar K)\e^{-2u(t)}-f)\e^{2u(t)}\partial_tu(t)d\mu_{\bar g}\\
&=\int_M((-\Delta_{\bar g}u(t)+\bar K)\e^{-2u(t)}-f)\partial_tu(t)d\mu_{g(t)}\\
&=\int_M(K_{g(t)}-f)\partial_tu(t)d\mu_{g(t)}=\int_M(K_{g(t)}-f+\alpha(t))\partial_tu(t)d\mu_{g(t)}\\
&=-\int_M|\partial_tu(t)|^2d\mu_{g(t)}\le0.
\end{split}
\end{equation}
Therefore, we have 
\begin{equation}\label{APrioriBound}
E_{f}(u(\tau))+\int_0^\tau\int_M|\partial_tu(t)|^2d\mu_{g(t)}dt= E_{f}(u(0)) \qquad \text{for $0 < \tau < T$.}
\end{equation}
\\
5. Since $u(t)\in \mathcal{C}_A$ for $t \in [0,T)$ by 1., we may use 4., \eqref{jensen-consequence} and \eqref{eq:C-A-energy-ineq} to observe that 
\begin{equation}\label{EstimateNabla1}
\begin{split}
\|\nabla_{\bar g}u(t)\|^2_{L^2(M,\bar g)}&=2E_{f}(u(t))-\int_M(2\bar Ku(t)-f\e^{2u(t)})d\mu_{\bar g}\\
&=2E_{f}(u(t))+\int_M(2|\bar K|u(t)+f\e^{2u(t)})d\mu_{\bar g}\\
&\le 2 E_{f}(u_0)+|\bar K|\log(A)+ A \|f\|_{L^\infty(M,\bar g)}\\
&\le \|\nabla u_0 \|_{L^2(M,\bar g)}^2 + |\bar K|\Bigl(\log(A)+ 2\|u_0\|_{L^1(M,\bar g)}\Bigr)  +  2 A \|f\|_{L^\infty(M,\bar g)}\\
&\le c_0 + c_1 \|f\|_{L^\infty(M,\bar g)} \qquad \text{for $t \in [0,T)$.}
\end{split}
\end{equation}
with constants $c_0,c_1>0$ depending only on $u_0$ (recall here that $A = \int_M \e^{2u_0(t)}d\mu_{\bar g}$).\\
6. With \eqref{Unigradbound} and \autoref{Onofri-ungleichung} we can estimate 
\begin{align*}
  A=\int_M\e^{2u(t)}d\mu_{\bar g}=\e^{2\bar u(t)}\int_M\e^{2(u(t)-\bar u(t))}d\mu_{\bar g} &\le\e^{2\bar u(t)}C_{\text{MT}}\exp(\eta_1\|\nabla_{\bar g}(2u(t))\|^2_{L^2(M,\bar g)})\\
  &\le\e^{2\bar u(t)}C_{\text{MT}}\exp \bigl(\eta_1(c_1 + c_2 \|f\|_{L^\infty(M,\bar g)})\bigr)
\end{align*}
and therefore
\[\bar u(t)\ge\frac12\log\left( \frac{A}{C_{\text{MT}}}\right) - \frac12\eta_1(c_1 + c_2 \|f\|_{L^\infty(M,\bar g)})= m_0 - m_1 \|f\|_{L^\infty(M,\bar g)}
\] 
with constants $m_0 \in \R$, $m_1 >0$ depending only on $u_0$. Combining this lower bound with the upper bound given by \eqref{Jensen-consequence}, we obtain \eqref{UniBarU}.\\
7.  With \autoref{Onofri-ungleichung}, \eqref{jensen-consequence}, and \eqref{EstimateNabla1} we directly get for any $p\in\R$ that
\begin{align*}
\int_M\e^{2pu(t)}d\mu_{\bar g}=\e^{2p\bar u(t)}\int_M\e^{2p(u(t)-\bar u(t))}d\mu_{\bar g} &\le
\e^{ p \log(A)}C_{\text{MT}}\exp(4\eta_2p^2\|\nabla_{\bar g}u(t)\|^2_{L^2(M,\bar g)})\\
&\le A^p C_{\text{MT}}\exp\Bigl((4\eta_2p^2\bigl(c_1 + c_2 \|f\|_{L^\infty(M,\bar g)}\bigr) \Bigr)\\
&\le C_{\text{MT}} A^{p}\e^{4\eta_2p^2 c_1} \exp \bigl(4\eta_2p^2 c_2 \|f\|_{L^\infty(M,\bar g)}\bigr)\\
&= \nu_0 \e^{\nu_1 \|f\|_{L^\infty(M,\bar g)}}
\end{align*} 
with constants $\nu_i=\nu_i(u_0,p)>0$, $i\in\{0,1\}$.
\end{proof}

\section{Main Results}\label{SectionMainResults}
In the following, we put 
\begin{equation}
  \label{eq:def-C-p-A}
\mathcal{C}_{p,A}:= W^{2,p}(M,\bar g) \cap \mathcal{C}_{p,A} =  \left\{v\in W^{2,p}(M,\bar g)\mid \int_M\e^{2v}d\mu_{\bar g}=A\right\} \qquad \text{for $p>2,$ $A>0$.}
\end{equation}
The following is our first main result.
\begin{theorem}\label{ShortTimeExistence}
  Let $f\in C^\infty(M)$, $p>2$, and $u_0\in \mathcal{C}_{p,A}$ for a given $A>0$.\\
Then the initial value problem \eqref{PCFVolNew1}, \eqref{PCFVolNew2}
admits a unique global solution
$$
u \in C([0,\infty)\times M)\cap C([0,\infty);H^1(M,\bar g))\cap C^\infty((0,\infty)\times M)
$$
satisfying the energy bound $E_{f}(u(t))\le E_{f}(u_0)$ for all $t \ge 0$.\\ 
Moreover, $u$ is uniformly bounded in the sense that
\[
  \sup_{t>0}\|u(t)\|_{L^\infty(M,\bar g)}<\infty.
\]
Furthermore, if $(t_l)_l \subset (0,\infty)$ is a sequence with $t_l \to \infty$ as $l \to \infty$, then, after passing to a subsequence, $u(t_l)$ converges in $H^2(M,\bar g)$ to a function $u_\infty \in H^2(M,\bar g) \cap \mathcal{C}_{A}$ solving the equation
\begin{equation}\label{PCP-f-lambda}
-\Delta_{\bar g}u_\infty + \bar K = f_\lambda \e^{2u_\infty} \qquad \text{in $M$,}
\end{equation}
where $f_\lambda:= f +\lambda$ with
\begin{equation}
  \label{eq:def-lambda}
\lambda=\frac1A\left(\bar K- \int_M f \e^{2u_\infty}d\mu_{\bar g}\right).
\end{equation}
In other words, $u_\infty$ induces a metric $g_\infty$ with $\vol_{g_\infty}= A$ and Gauss curvature $K_{g_\infty}$ satisfying
\begin{equation}
  \label{eq:def-f-lambda}
K_{g_\infty}(x)= f_{\lambda}(x)=f(x)+\lambda\quad\text{for}\quad x\in M.
\end{equation}
\end{theorem}

Some remarks are in order.

\begin{remark}
  \label{convergence-remark}
  It follows in a standard way that, under the assumptions of \autoref{ShortTimeExistence}, the $\omega$-limit set
  $$
  \omega(u_0):= \bigcap_{T>0} \overline{ \{u(t)\::\: T \le t < \infty\}}
  $$
  is a compact connected subset of $H^2(M,\bar g) \cap \mathcal{C}_{A}$ (with respect to the $H^2$-topology) consisting of solutions of \eqref{PCP-f-lambda}, \eqref{eq:def-lambda}, which are precisely the critical points of the restriction of the energy functional $E_f$ to $\mathcal{C}_{A}$.
  
In particular, the connectedness implies that, if $u_\infty$ in \autoref{ShortTimeExistence} is an isolated critical point in $\mathcal{C}_{A}$, then $\omega(u_0)= \{u_\infty\}$ and therefore we have the full convergence of the flow line
\begin{equation}
  \label{eq:fullconvergence-flow-line}
u(t) \to u_\infty \quad \text{in $H^2(M,\bar g)$}\quad \text{as $t \to \infty$.}
\end{equation}
In particular, \eqref{eq:fullconvergence-flow-line} holds if $u_\infty$ is a strict local minimum of the restriction of $E_f$ to $\mathcal{C}_{A}$.
\end{remark}

\begin{remark}
For functions $f<0$, the convergence of the flow \eqref{PCFnegative} is shown in \cite{Ho11}. For the additively rescaled flow \eqref{PCFVolNew1} with initial data \eqref{PCFVolNew2} we get convergence for arbitrary functions $f\in C^\infty(M)$. In general we do not have any information about $\lambda$ and therefore no information about the sign of $f_{\lambda}$ in \autoref{ShortTimeExistence}. On the other hand, more information can be derived for certain functions $f \in C^\infty(M)$ and certain values of $A>0$.

\begin{enumerate}
\item[(i)] In the case where $A \le - \frac{\bar K}{\|f\|_{L^\infty(M,\bar g)}}$, it follows that  
$$
\lambda= \frac1A\left(\bar K- \int_M f \e^{2u}d\mu_{\bar g}\right) \le \frac{\bar K}{A} + \frac{\|f\|_{L^\infty(M,\bar g)}}{A} \int_M \e^{2u}d\mu_{\bar g}
= \frac{\bar K}{A} + \|f\|_{L^\infty(M,\bar g)} \le 0
$$
for every solution $u \in \mathcal{C}_{2,A}:=\left\{v\in H^2(M,\bar g)\mid\int_M\e^{2v}d\mu_{\bar g}=0\right\}$ of the static problem \eqref{PCP-f-lambda}, and therefore this also applies to $\lambda$ in \autoref{ShortTimeExistence} in this case.

\item[(ii)] The following theorems show that $f_{\lambda}$ in \autoref{ShortTimeExistence} may change sign if
$A > - \frac{\bar K}{\|f\|_{L^\infty(M,\bar g)}}$, so in this case we get a solution of the static problem \eqref{PCP} for sign-changing functions $f\in C^\infty(M)$ by using the additively rescaled prescribed Gauss curvature flow \eqref{PCFVolNew1}.
\end{enumerate}
\end{remark}

\begin{theorem}
  \label{sign-changing}
  Let $p>2$. For every $A>0$ and $c> - \frac{\bar K}{A}$ there exists $\varepsilon= \varepsilon(c,A,\bar K)>0$ with the following property.

  If  $u_0 \equiv \frac{1}{2}\log(A) \in \mathcal{C}_{p,A}$ and $f\in C^\infty(M)$
  with $-c \le f \le 0$ and $\|f+c\|_{L^1(M,\bar g)} < \varepsilon$ are chosen in \autoref{ShortTimeExistence}, then the value $\lambda$ defined in \eqref{eq:def-lambda} is positive. 

In particular, if $f$ has zeros on $M$, then $f_\lambda$ in \eqref{eq:def-f-lambda} is sign changing.
\end{theorem}

Under fairly general assumptions on $f$, we can prove that $\lambda>0$ if $A$ is sufficiently large and $u_0 \in \mathcal{C}_{p,A}$ is chosen suitably. 

\begin{theorem}
\label{sec:stat-minim-probl-1-cor-2-theorem}
Let $f \in C^\infty(M)$ be nonconstant with $\max_{x\in M} f(x) = 0$. Then there exists $\kappa>0$ with the property that for every $A \ge \kappa$ there exists $u_0 \in \mathcal{C}_{p,A}$ such that the value $\lambda$ defined in \eqref{eq:def-lambda} is positive.
\end{theorem}

In fact we have even more information on the associated limit $u_\infty$ in this case, see \autoref{sec:stat-minim-probl-1-cor-2} below.

It remains open how large $\lambda$ can be depending on $A$ and $f$. The only upper bound we have is 
\begin{equation}\label{Conditionlambda1}
\lambda<- \int_M f d\mu_{\bar g},
\end{equation}
since we must have 
\[
\bar f_\lambda=\frac{1}{\vol_{\bar g}}\int_M f_\lambda d\mu_{\bar g}=\int_M f d\mu_{\bar g}+\lambda\overset{!}{<}0,
\]
so that $f_\lambda$ fulfills the necessary condition \eqref{Condition2} provided by Kazdan and Warner in \cite{KazWar74_1}.

\section{The static Minimisation Problem with Volume Constraint}
\label{sec:stat-minim-probl}
To obtain additional information on the limiting function $u_\infty$ and the value $\lambda \in \R$ associated to it by \eqref{eq:def-lambda} and \eqref{eq:def-f-lambda}, we need to consider the associated static setting for the prescribed Gauss curvature problem with the additional condition of prescribed volume. In this setting, we wish to find, for given $f \in C^\infty(M)$ and $A>0$, critical points of the restriction of the functional $E_{f}$ defined in \eqref{Ef} to the set $\mathcal{C}_{A}$ defined in \eqref{NB}. A critical point $u \in \mathcal{C}_{A}$ of this restriction is
a solution of \eqref{PCP-f-lambda} for some $\lambda \in \R$, where, here and in the following, we put again
$f_\lambda:= f + \lambda \in C^\infty(M)$. In other words, such a critical point induces, similarly as the limit $u_\infty$ in \autoref{ShortTimeExistence}, a metric $g^u$ with Gauss curvature $K_{g^u}$ satisfying $K_{g^u}(x)= f_{\lambda}(x)=f(x)+\lambda.$
The unknown $\lambda \in \R$ arises in this context as a Lagrange multiplier and is a posteriori characterised again by 
$$
\lambda=\frac1A\left(\bar K- \int_M f \e^{2u}d\mu_{\bar g}\right).
$$

In the study of critical points of the restriction of $E_{f}$ to $\mathcal{C}_{A}$, it is natural to consider the minimisation problem first. For this we set
$$
m_{f,A} = \inf_{u \in \mathcal{C}_{A}}E_f(u).
$$
We have the following estimates for $m_{f,A}$:
\begin{lemma}
\label{m-A-est}
Let $f \in C^\infty(M)$, $A>0$. Then we have
\begin{equation}
\label{eq:m-A-first-estimate}
m_{f,A} \le \frac{1}{2} \left(\bar K \log(A) - A\int_{M} f d \mu_{\bar g}\right).  \end{equation}
Moreover, if $\max f \ge 0$, then we have
\begin{equation}
\label{eq:m-A-second-estimate}
\limsup_{A \to \infty} \frac{m_{f,A}}{A} \le 0.
\end{equation}
\end{lemma}

\begin{proof}
Let $u_0(A)\equiv\frac12\log(A)$, so that $\int_M\e^{2u_0(A)}d\mu_{\bar g}=A$. Hence $u_0(A)$ is the (unique) constant function in $\mathcal{C}_{A}$, and 
\begin{align*}
m_{f,A}&\le E_{f}(u_0(A))=\frac12\int_M(|\nabla_{\bar g}u_0(A)|^2_{\bar g}+2\bar Ku_0(A)-f\e^{2u_0(A)})d\mu_{\bar g}\\
&=\frac12\int_M(\bar K\log(A)-f A)d\mu_{\bar g} = \frac{1}{2} \left(\bar K \log(A) - A\int_{M} f d \mu_{\bar g}\right).
\end{align*}
This shows \eqref{eq:m-A-first-estimate}. To show \eqref{eq:m-A-second-estimate}, we let $\varepsilon>0$. Since $f \in C^\infty(M)$ and $\max f \ge 0$ by assumption, there exists an open set $\Omega \subset M$ with $f \ge -\varepsilon$ on $\Omega$. Next, let $\psi \in C^\infty(M)$, $\psi \ge 0$, be a function supported in $\Omega$ and with
$\|\psi\|_{L^\infty(M,\bar g)} = 2$. Consequently, the set $\Omega':= \{x \in M\mid \psi >1\}$ is a nonempty open subset of $\Omega$, and therefore $\mu_{\bar g}(\Omega')>0$.

Next we consider the continuous function
$$
h: [0,\infty) \to [0,\infty);\quad h(\tau)= \int_M\e^{2 \tau \psi}d\mu_{\bar g}
$$
and we note that $h(0)= \int_M d\mu_{\bar g}=1$, and that
$$
h(\tau) \ge \int_{\Omega'} \e^{2 \tau \psi}d\mu_{\bar g} \ge \e^{2\tau} \mu_{\bar g}(\Omega') \qquad \text{for $\tau \ge 0$.}
$$
Hence for every $A \ge 1$ there exists
\begin{equation}
  \label{eq:tau-A}
0 \le \tau_A  \le \frac{1}{2}\Bigl(\log(A) -\log (\mu_{\bar g}(\Omega'))\Bigr)
\end{equation}
with $h(\tau_A)=A$ and therefore $\tau_A \psi \in \mathcal{C}_{A}$. Consequently,
\begin{align*}
  m_{f,A} &\le E_f(\tau_A \psi) = \frac12 \int_M(|\nabla_{\bar g}\tau_A \psi|^2_{\bar g}+2\bar K \tau_A \psi-f\e^{2\tau_A \psi})d\mu_{\bar g}\\
          &= \tau_A^2 c_1 - \tau_A c_2 -c_3 - \frac{1}{2} \int_{\Omega} f\e^{2\tau_A \psi}d\mu_{\bar g}
\end{align*}
with
$$
c_1 = \frac{1}{2}\int_M |\nabla_{\bar g}\psi|^2_{\bar g}d\mu_{\bar g},\quad c_2 = - \bar K \int_{M} \psi d\mu_{\bar g}  \quad \text{and}\quad c_3 =  \frac{1}{2} \int_{M \setminus \Omega} f d\mu_{\bar g}.
$$
Since $f \ge -\varepsilon$ on $\Omega$, we thus deduce that
$$
m_{f,A} \le \tau_A^2 c_1 -2 \tau_A c_2 +c_3 + \frac{\varepsilon}{2} \int_{\Omega} \e^{2\tau_A \psi}d\mu_{\bar g} \le \tau_A^2 c_1 -2 \tau_A c_2 +c_3 +\frac{\varepsilon A}{2}.
$$
Since $\frac{\tau_A}{A} \to 0$ as $A \to \infty$ by \eqref{eq:tau-A}, we conclude that
$$
\limsup_{A \to \infty} \frac{m_{f,A}}{A} \le \frac{\varepsilon}{2}.
$$
Since $\varepsilon>0$ was chosen arbitrarily, \eqref{eq:m-A-second-estimate} follows. 
\end{proof}

\begin{lemma}
\label{lemma-lagrange-multiplier}
Let $f \in C^\infty(M)$ nonconstant with $\max_{x\in M} f(x) = 0$.  For every $\varepsilon>0$ there exists $\kappa_0>0$ with the following property. If $A \ge \kappa_0$ and $u \in \mathcal{C}_{A}$ is a solution of
\begin{equation}
\label{eq:PCP-lambda}
-\Delta_{\bar g}u + \bar K = (f+\lambda) \e^{2u}
\end{equation}
for some $\lambda \in \R$ with $E_f(u)< \frac{\varepsilon A}{2}$, then we have $\lambda<\varepsilon$. 
\end{lemma}

\begin{proof}
For given $\varepsilon>0$, we may choose $\kappa_0>0$ sufficiently large so that
$\frac{|\bar K|}{2} \frac{\log(A)}{|A|}< \frac{\varepsilon}{2}$ for $A \ge \kappa_0$.

Now, let $A \ge \kappa_0$, and let $u \in \mathcal{C}_{A}$ be a solution of \eqref{eq:PCP-lambda} satisfying $E_f(u)< \frac{\varepsilon A}{2}$.
Integrating \eqref{eq:PCP-lambda} over $M$ with respect to $\mu_{\bar g}$ and using that $\vol_{\bar g}(M)=1$ and $\int_{M} \e^{2u}d\mu_{\bar g}=A$, we obtain
\begin{align*}
\lambda &= \frac{1}{A}\left(\bar K - \int_{M} f \e^{2u}d\mu_{\bar g}\right) \le -  \frac{1}{A}\int_{M} f \e^{2u}d\mu_{\bar g} \\
&= \frac{1}{A} \left(E_{f}(u)-\frac12\int_M(|\nabla_{\bar g}u|^2_{\bar g}+2\bar K u)d\mu_{\bar g} \right)\le \frac{1}{A}\left( E_f(u) +|\bar K| \bar u\right)\\
&\le \frac{\varepsilon }{2} + \frac{|\bar K|}{2} \frac{\log(A)}{A} <\varepsilon,
\end{align*}
as claimed. Here we used \eqref{Jensen-consequence} to estimate $\bar u$.
\end{proof}

\begin{proposition}
\label{local-minimizers-sequence}  
Let $f \in C^\infty(M)$ be a nonconstant function with $\max_{x\in M} f(x) =0$.  Moreover, let 
$\lambda_n \to 0^+$ for $n\to\infty$, and let $(u_n)_{n\in\N}$ be a sequence of solutions of
\begin{equation}
\label{eq1:parameter-dep}
-\Delta_{\bar g} u_n + \bar K  = (f +\lambda_n)\e^{2u_n} \quad \text{in $M$}
\end{equation}
which are weakly stable in the sense that
\begin{equation}
\label{eq:weakly-stable}
\int_{M}(|\nabla_{\bar g} h|_{\bar g}^2 -2 (f+\lambda_n) \e^{2u_n} h^2) d \mu_{\bar g} \ge 0 \quad \text{for all $h \in H^1(M)$.}   
\end{equation}
Then $u_n \to u_0$ in $C^2(M)$, where $u_0$ is the unique solution of
\begin{equation}
\label{eq:parameter-dep-limit}
-\Delta_{\bar g} u_0 + \bar K  = f \e^{2 u_0} \qquad \text{in $M$.}
\end{equation}
\end{proposition}

\begin{proof}
We only need to show that
\begin{equation}
\label{eq:sufficient-bounded}
\text{$(u_n)_{n\in\N}$ is bounded in $C^{2,\alpha}(M)$ for some $\alpha>0$.} 
\end{equation}
Indeed, assuming this for the moment, we may complete the argument as follows.
Suppose by contradiction that there exists $\varepsilon>0$ and a subsequence, also denoted by $(u_n)_{n\in\N}$, with the property that
\begin{equation}
\label{eq:sufficient-bounded-contradiction}
\|u_n-u_0\|_{C^2(M)} \ge \varepsilon \qquad \text{for all $n \in \N$.}
\end{equation}
By \eqref{eq:sufficient-bounded} and the compactness of the embedding $C^{2,\alpha}(M) \hookrightarrow C^2(M)$, we may then pass to a subsequence, still denoted by $(u_n)_{n\in\N}$, with $u_n \to u_*$ in $C^2(M)$ for some $u_* \in C^2(M)$. Passing to the limit in \eqref{eq1:parameter-dep}, we then see that $u_*$ is a solution of \eqref{eq:parameter-dep-limit}, which by uniqueness implies that $u_* = u_0$.
This contradicts \eqref{eq:sufficient-bounded-contradiction}, and thus the claim follows.  

The proof of \eqref{eq:sufficient-bounded} follows by similar arguments as in \cite[p.~1063 f.]{DinLiu95}. Since the framework is slightly different, we sketch the main steps here for the convenience of the reader.  We first note that, by the same argument as in \cite[p.~1063 f.]{DinLiu95}, there exists a constant $C_0>0$ with
\begin{equation}
\label{eq:proof-lower-bound}
u_n \ge -C_0 \qquad \text{for all $n$.}   
\end{equation}

Since $\{f < 0\}$ is a nonempty open subset of $M$ by assumption, we may fix a nonempty open subdomain $\Omega \subset \subset \{f < 0\}$. By \cite[Appendix]{BorGalStr15}, there exists a constant $C_1>0$ with 
$$
\|u_n^+\|_{H^1(\Omega,\bar g)} \le C_1 \qquad \text{for all $n$} 
$$
and therefore
\begin{equation}
\label{eq:v-n-omega-est} \int_{\Omega}\e^{2u_n}d\mu_{\bar g} \le \int_{\Omega}\e^{2u_n^+}d\mu_{\bar g} \le C_2 \qquad \text{for all $n$}
\end{equation}
for some $C_2>0$ by the Moser--Trudinger inequality.  Next, we consider a nontrivial, nonpositive function $h \in C^\infty_c(\Omega) \subset C^\infty(M)$ and the unique solution $w \in C^\infty(M)$ of the equation
$$
-\Delta_{\bar g} w + \bar K  = h\e^{2w} \quad \text{in $M$.}
$$
Moreover, we let $w_n:= u_n - w$, and we note that $w_n$ satisfies 
$$
-\Delta_{\bar g} w_n +h\e^{2w}  = (f+{\lambda_n})\e^{2u_n} \quad \text{in $M$.}
$$
Multiplying this equation by $\e^{2w_n}$ and integrating by parts, we obtain
\begin{align}
\int_{M}(f+\lambda_n)\e^{2(u_n+w_n)} d\mu_{\bar g} &=   \int_{M}\Bigl(-\Delta_{\bar g} w_n + h \e^{2w} \Bigr)\e^{2w_n} d\mu_{\bar g}=  \int_{M} \Bigl(2\e^{2w_n} |\nabla_{\bar g}w_n|_{\bar g}^2 +
                                                      h \e^{2(w+w_n)} \Bigr) d\mu_{\bar g} \nonumber\\
&= 2 \int_{M}|\nabla_{\bar g} \e^{w_n}|_{\bar g}^2 d \mu_{\bar g} +  \int_{\Omega} h \e^{2u_n} d\mu_{\bar g}.  \label{eq:proof-dl-1}
\end{align}
Moreover, applying \eqref{eq:weakly-stable} to $h= \e^{w_n}$ gives 
\begin{equation}
\label{eq:proof-dl-2}
\int_{M}(f+\lambda_n)\e^{2(u_n+w_n)} d\mu_{\bar g} \le  \frac{1}{2} \int_{M}|\nabla_{\bar g} \e^{w_n}|_{\bar g}^2 d \mu_{\bar g}.
\end{equation}
Combining   \eqref{eq:v-n-omega-est}, \eqref{eq:proof-dl-1} and \eqref{eq:proof-dl-2} yields 
\begin{equation}
  \label{eq:inverse-poincare}
  \|\nabla_{\bar g} \e^{w_n}\|_{L^2(M,\bar g)}^2 \le -\frac{2}{3} \int_{\Omega} h \e^{2u_n} d\mu_{\bar g}
\le \frac{2}{3}\|h\|_{L^\infty(M,\bar g)}C_2 \quad \text{for all $n$.} 
\end{equation}
Next we claim that also $\|\e^{w_n}\|_{L^2(M,\bar g)}$ remains uniformly bounded. 
Suppose by contradiction that
\begin{equation}
  \label{eq:contradiction-dl}
\|\e^{w_n}\|_{L^2(M,\bar g)} \to \infty \qquad \text{as $n \to \infty$.}
\end{equation}
We then set $v_n:= \frac{\e^{w_n}}{\|\e^{w_n}\|_{L^2(M,\bar g)}}$, and we note that
\begin{equation}
  \label{eq:limit-v-n-1}
\|v_n\|_{L^2(M,\bar g)} = 1 \quad \text{for all $n$}\quad  \text{and}\quad \|\nabla_{\bar g} v_n\|_{L^2(M,\bar g)}^2  \to 0  \quad \text{as $n \to \infty$}
\end{equation}
by \eqref{eq:inverse-poincare}. Consequently, we may pass to a subsequence satisfying $v_n \rightharpoonup v$ in $H^1(M,\bar g)$, where $v$ is a constant function with
\begin{equation}
  \label{eq:norm-limit}
\|v\|_{L^2(M,\bar g)}=1.  
\end{equation}
However, since
$$
\|\e^{w_n}\|_{L^2(\Omega,\bar g)} \le \|\e^{u_n}\|_{L^2(\Omega,\bar g)} \|\e^{-w}\|_{L^\infty(\Omega,\bar g)} 
\le \sqrt{C_2} \|\e^{-w}\|_{L^\infty(\Omega,\bar g)}\quad \text{for all $n \in \N$}
$$
by \eqref{eq:v-n-omega-est} and therefore 
$$
\|v\|_{L^2(\Omega,\bar g)} = \lim_{n \to \infty}\|v_n\|_{L^2(\Omega,\bar g)}=
\lim_{n \to \infty}\frac{\|\e^{w_n}\|_{L^2(\Omega,\bar g)}}{\|\e^{w_n}\|_{L^2(M,\bar g)}} = 0
$$
by \eqref{eq:contradiction-dl}, we conclude that the constant function $v$ must vanish identically, contradicting \eqref{eq:norm-limit}.

Consequently, $\|\e^{w_n}\|_{L^2(M,\bar g)}$ remains uniformly bounded, which by \eqref{eq:inverse-poincare} implies that $\e^{w_n}$ remains bounded in $H^1(M,\bar g)$ and therefore in $L^p(M,\bar g)$ for any $p < \infty$. Since $\e^{u_n} \le \|\e^{w}\|_{L^\infty(M,\bar g)} \e^{w_n}$ on $M$ for all $n \in \N$,
it thus follows that also $\e^{u_n}$ remains bounded in $L^p(M,\bar g)$ for any $p< \infty$. Moreover, by \eqref{eq:proof-lower-bound}, the same applies to the sequence $u_n$ itself. Therefore, applying successively elliptic $L^p$ and Schauder estimates to \eqref{eq1:parameter-dep}, we deduce \eqref{eq:sufficient-bounded}, as required.
\end{proof}

In the proof of the next proposition, we need the following classical interpolation inequality, see e.g.~\cite{CecMon08}.
\begin{lemma}[Gagliardo--Nirenberg--Lady\v{z}henskaya inequality]\label{eindSob}
For every $r>2$, there exists a constant $C_{\text{GNL}}=C_{\text{GNL}}(r)>0$ with 
\[\|\zeta\|^{r}_{L^r(M,\bar g)} \le C_{\text{GNL}}\|\zeta\|^{2}_{L^2(M,\bar g)}\|\zeta\|^{r-2}_{H^1(M,\bar g)} \qquad \text{for every $\zeta\in H^1(M,\bar g)$.} \]
\end{lemma}

\begin{proposition}
  \label{sec:stat-minim-probl-2}
Let $f \in C^\infty(M)$ be a nonconstant function with $\max_{x\in M} f(x) =0$. Then there exists $\lambda_\sharp$ and a $C^1$-curve $(-\infty,\lambda_{\sharp}] \to C^2(M); \quad \lambda \mapsto u_\lambda$ with the following properties.
  \begin{enumerate}
  \item[(i)] If $\lambda \le 0$, then $u_\lambda$ is the unique solution of
\begin{equation}
  \label{eq:parameter-dep-cor}
-\Delta_{\bar g} u + \bar K  = f_\lambda \e^{2 u} \quad \text{in $M$}
\end{equation}
and a global minimum of $E_{f_\lambda}$.
\item[(ii)] If $\lambda \in (0,\lambda_\sharp]$, then $u_\lambda$ is the unique weakly stable solution of \eqref{eq:parameter-dep-cor} in the sense of \eqref{eq:weakly-stable}, and it is a local minimum of $E_{f_\lambda}$. 
\item[(iii)] The curve of functions $\lambda \mapsto u_\lambda$ is pointwisely strictly increasing on $M$, and so the volume function 
  \begin{equation}
    \label{eq:def-volume}
(-\infty,\lambda_\sharp] \to [0,\infty);\quad  \lambda \mapsto V(\lambda):= \int_{M}\e^{2u_\lambda} d\mu_{\bar g}
  \end{equation}
  is continuous and strictly increasing. 
\end{enumerate}
\end{proposition}

\begin{proof}
  We already know that, for $\lambda \le 0$, the energy $E_{f_\lambda}$ admits a strict global minimiser $u_\lambda$ which depends smoothly on $\lambda$. Moreover, by \cite[Proposition 2.4]{BorGalStr15}, the curve $\lambda \mapsto u_\lambda$ can be extended as a $C^1$-curve to an interval $(-\infty,\lambda_{\sharp}]$ for some $\lambda_{\sharp}>0$. We also know from \cite[Proposition 2.4]{BorGalStr15} that, for $\lambda \in (-\infty,\lambda_{\sharp}]$, the solution $u_\lambda$ is strongly stable in the sense that 
\begin{equation}
  \label{eq:stable}
C_\lambda := \inf_{h \in H^1(M,\bar g)} \frac{1}{\|h\|_{H^1(M,\bar g)}^2}\int_{M}\Bigl(|\nabla_{\bar g} h|_{\bar g}^2 -2 f_\lambda \e^{2u_\lambda} h^2\Bigr) d \mu_{\bar g} >0.  
\end{equation}
Here we note that the function $\lambda \mapsto C_\lambda$ is continuous since $u_\lambda$ depends continuously on $\lambda$ with respect to the $C^2$-norm. Next we prove that, after making $\lambda_\sharp>0$ smaller if necessary, the function $u_\lambda$ is the unique weakly stable solution of \eqref{eq:parameter-dep-cor} for $\lambda \in (0,\lambda_{\sharp}]$. Arguing by contradiction, we assume that there exists a sequence 
$\lambda_n \to 0^+$ and corresponding weakly stable solutions $(u_n)_{n\in\N}$  of
\begin{equation}
  \label{eq:parameter-dep}
-\Delta_{\bar g} u_n + \bar K  = (f + \lambda_n)\e^{2u_n} \quad \text{in $M$}
\end{equation}
with the property that $u_n \not = u_{\lambda_n}$ for every $n \in \N$. By \autoref{local-minimizers-sequence}, we know that $u_n \to u_0$ in $C^2(M)$. Consequently, $v_n := u_n -  u_{\lambda_n} \to 0$ in $C^2(M)$ as $n \to \infty$, whereas the functions $v_n$ solve
\begin{equation}
  \label{eq:parameter-dep-v-n}
  -\Delta_{\bar g} v_n   = (f + \lambda_n)\bigl(\e^{2u_n}-\e^{2u_{\lambda_n}}\bigr)=
  (f + \lambda_n)\e^{2u_{\lambda_n}} \bigl(\e^{2v_n}-1\bigr)
  \quad \text{in $M$} \quad \text{for every $n \in \N$.}
\end{equation}
Combining this fact with \eqref{eq:stable}, we deduce that 
\begin{align*}
  \|v_n\|_{H^1(M,\bar g)}^2 &\le \frac{1}{C_\lambda} \int_{M}\Bigl(|\nabla_{\bar g} v_n|_{\bar g}^2 -2 (f+\lambda_n) \e^{2u_{\lambda_n}} v_n^2\Bigr)d\mu_{\bar g}\\
&= \frac{1}{C_\lambda} \int_{M}(f+\lambda_n) \e^{2u_{\lambda_n}} \bigl(\e^{2v_n}-1- 2v_n\Bigr)v_nd\mu_{\bar g}.  
\end{align*}
Since $v_n \to 0$ in $C^2(M)$, there exists a constant $C>0$ with $|(\e^{2v_n}-1- 2v_n)v_n| \le C |v_n|^3$ on $M$ for all $n \in \N$, which then implies with H\"older's inequality and \autoref{eindSob} that
\begin{align*}
\|v_n\|_{H^1(M,\bar g)}^2&\le C \|(f+\lambda_n) \e^{2u_{\lambda_n}}\|_{L^\infty(M,\bar g)}\|v_n\|_{L^3(M,\bar g)}^3\\
 &\le C\left(\int_M|v_n|^{3\cdot\frac43}d\mu_{\bar g}\right)^{\frac34}= C \|v_n\|_{L^4(M,\bar g)}^3\le C \|v_n\|_{H^1(M,\bar g)}^3
\end{align*}
with a constant $C>0$ independent on $M$. This contradicts the fact that $v_n \to 0$ in $H^1(M)$ as $n \to \infty$. The claim thus follows.

It remains to prove that the curve of functions $\lambda \mapsto u_\lambda$ is pointwisely strictly increasing on $M$. This is a consequence of the uniqueness of weakly stable solutions stated in (ii) and the fact that, as noted in \cite{DinLiu95}, if $u_{\lambda_0}$ is a solution for some $\lambda_0 \in (-\infty,\lambda_{\sharp}]$, it is possible to construct, via the method of sub- and supersolutions, for every $\lambda < \lambda_0$, a {\em weakly stable} solution $u_\lambda$ with $u_\lambda < u_{\lambda_0}$ everywhere in $M$.   
\end{proof}

\begin{corollary}
  \label{sec:stat-minim-probl-1-cor-1}
Let $f \in C^\infty(M)$ be nonconstant with $\max_{x\in M} f (x)= 0$, and let $\lambda_\sharp>0$ be given as in \autoref{sec:stat-minim-probl-2}. Then there exists $\kappa_1>0$ with the following property.

If $A \ge \kappa_1$ and $u \in \mathcal{C}_{A}$ is a solution of
  \begin{equation}
    \label{eq:PCP-lambda-corollary}
  -\Delta_{\bar g}u + \bar K = (f+\lambda) \e^{2u}
  \end{equation}
  for some $\lambda \in \R$ with $E_f(u)< \frac{\lambda_\sharp A}{2}$, then $0< \lambda < \lambda_\sharp$, and $u$ is not a weakly stable solution of \eqref{eq:PCP-lambda-corollary}, so $u \not = u_\lambda$.
\end{corollary}

\begin{proof}
  Let $\kappa_0>0$ be given as in \autoref{lemma-lagrange-multiplier} for $\varepsilon = \lambda_\sharp>0$. Moreover, let
  $$
  \kappa_1:= \max \left \{ \kappa_0, V(u_{\lambda_\sharp})\right\}
  $$
  with $V$ defined in \eqref{eq:def-volume}. Next, let $u \in \mathcal{C}_{A}$ be a solution of \eqref{eq:PCP-lambda-corollary} for some $\lambda \in \R$ with $E_f(u)< \frac{\lambda_\sharp A}{2}$. From \autoref{lemma-lagrange-multiplier}, we then deduce that $0< \lambda < \lambda_\sharp$, and by \autoref{sec:stat-minim-probl-2} (iii) we have $u \not = u_\lambda$. Since $u_\lambda$ is the unique weakly stable solution of \eqref{eq:PCP-lambda-corollary}, it follows that $u$ is not weakly stable.
\end{proof}

\begin{corollary}
  \label{sec:stat-minim-probl-1-cor-2}
   Let $p>2$, $f \in C^\infty(M)$ be nonconstant with $\max_{x\in M} f (x)= 0$, and let $\lambda_\sharp>0$ be given as in \autoref{sec:stat-minim-probl-2}. Then there exists $\kappa>0$ with the property that for every $A \ge \kappa$ the set
   \[\tilde{\mathcal{C}}:=\left\{u_0\in\mathcal{C}_{p,A} \mid E_f(u_0)<\frac{\lambda_\sharp A}{2}\right\}\]   
   is nonempty, and for every $u_0\in\tilde{\mathcal{C}}$ the global solution $u \in C([0,\infty)\times M)\cap C([0,\infty);H^1(M,\bar g))\cap C^\infty((0,\infty)\times M)$ of the initial value problem \eqref{PCFVolNew1}, \eqref{PCFVolNew2} converges, as $t \to \infty$ suitably, to a solution $u_\infty$ of the static problem \eqref{eq:PCP-lambda-corollary} for some $\lambda \in (0,\lambda_\sharp)$ which is not weakly stable and hence no local minimiser of $E_{f_\lambda}$. 
\end{corollary}

\begin{proof}
  Let $\kappa_1>0$ be given by \autoref{sec:stat-minim-probl-1-cor-1}. By \eqref{eq:m-A-second-estimate}, there exists $\kappa\ge\kappa_1>0$ with $m_{f,A}<\frac{\lambda_\sharp A}{4}$ for fixed $A>\kappa$. Consequently, there exists $u_0 \in \mathcal{C}_{A}\cap W^{2,p}(M,\bar g)$ with $E_f(u_0)< \frac{\lambda_\sharp A}{2}$. By \autoref{ShortTimeExistence}, the global solution $u \in C([0,\infty)\times M)\cap C([0,\infty);H^1(M,\bar g))\cap C^\infty((0,\infty)\times M)$ of the initial value problem \eqref{PCFVolNew1}, \eqref{PCFVolNew2} converges, as $t \to \infty$ suitably, to a solution $u_\infty \in \mathcal{C}_{A}$ of the static problem \eqref{eq:PCP-lambda-corollary} for some $\lambda \in \R$, whereas $E_f(u_\infty) \le E_f(u_0) < \frac{\lambda_\sharp A}{2}$. Consequently, 
$\lambda \in (0,\lambda_\sharp)$ by \autoref{sec:stat-minim-probl-1-cor-1}, and $u_\infty$ is not weakly stable.
\end{proof}

\section{Proof of the Main Results} 
\label{sec:proof-main-results}
\subsection{Preliminaries}
In the following, we consider, for fixed $T>0$, the spaces
\[L^p_tL^r_x:=L^p([0,T];L^r(M,\bar g))\quad\text{and}\quad L^p_tH^q_x:=L^p([0,T];H^q(M,\bar g)).\]
We stress that, although these spaces depend on $T$, we prefer to use a $T$-independent notation. We also note that, since $T<\infty$ and $\vol_{\bar g}=1$, we have $L^q_tL^r_x\subset L^s_tL^p_x$ for $p,q,r,s\in[1,\infty]$ with  $q\ge s$, $r\ge p$.

\begin{lemma}[Sobolev inequality]\label{SobolevLemma}
There exists a constant $C_S>0$ such that for every $T \le 1$ and every $\rho\in L^\infty_t H^1_x$ we have
\begin{equation}\label{Sobolev}
\|\rho\|^2_{L^4_tL^4_x}\le C_S(\|\rho\|^2_{L^\infty_tL^2_x}+\|\nabla_{\bar g}\rho\|^2_{L^2_tL^2_x})<\infty.
\end{equation}
\end{lemma}

\begin{proof}
By \autoref{eindSob}, applied with $r=4$, there exists a constant $C_{\text{GNL}}=C_{\text{GNL}}(4)>0$ with the property that, for all $T\le1$, 
\begin{align*}
\|\rho\|^4_{L^4_tL^4_x}&=\int_0^T\|\rho(t)\|^4_{L^4(M,\bar g)}dt \le C_{\text{GNL}}\int_0^T\|\rho(t)\|^2_{L^2(M,\bar g)}\|\rho(t)\|^2_{H^1(M,\bar g)}dt\\
&\le C_{\text{GNL}}\|\rho\|^2_{L^\infty_tL^2_x}\int_0^T(\|\rho(t)\|_{L^2(M,\bar g)}^2+\|\nabla_{\bar g}\rho(t)\|^2_{L^2(M,\bar g)})dt\\
&\le C_{\text{GNL}}\cdot T\:\|\rho\|^4_{L^\infty_tL^2_x}+C_{\text{GNL}}\|\rho\|^2_{L^\infty_tL^2_x}\|\nabla_{\bar g}\rho\|^2_{L^2_tL^2_x}\\
&\le C_{\text{GNL}}\left(\|\rho\|^4_{L^\infty_tL^2_x}+\|\rho\|^2_{L^\infty_tL^2_x}\|\nabla_{\bar g}\rho\|^2_{L^2_tL^2_x}\right) \le C_{\text{GNL}}\left(\frac{3}{2}\|\rho\|^4_{L^\infty_tL^2_x}+\frac{1}{2}\|\nabla_{\bar g}\rho\|^4_{L^2_tL^2_x}\right)\\
&\le \frac{3 C_{\text{GNL}}}{2}\left(\|\rho\|^2_{L^\infty_tL^2_x}+\|\nabla_{\bar g}\rho\|^2_{L^2_tL^2_x}\right)^2.
\end{align*}
Hence the first inequality in \eqref{Sobolev} holds with $C_S = \Bigl(\frac{3 C_{\text{GNL}}}{2}\Bigr)^{\frac{1}{2}}$. Moreover, since $T$ is finite, $\rho\in L^\infty_tH^1_x$ implies that $\rho\in L^p_tH^1_x$ for all $p\in[1,\infty]$ which shows that the RHS in \eqref{Sobolev} is finite.
\end{proof}

Now we can turn to the proofs of the main results.

\subsection{Short-Time Existence}\label{SectionShortTimeExistence}
Let $A>0$ and $p>2$ be fixed. We are looking for a short-time solution of \eqref{PCFVolNew1}, \eqref{PCFVolNew2} with initial value $u_0 \in \mathcal{C}_{p,A}$, where $\mathcal{C}_{p,A}$ is defined in  \eqref{eq:def-C-p-A}.
Using the Gauss equation \eqref{GaussEquation} we can rewrite \eqref{PCFVolNew1}, \eqref{PCFVolNew2} in the following way:
\begin{align}
\partial_tu(t)&=f-K_{g(t)}-\alpha(t)\nonumber\\
              &=\e^{-2u(t)}\Delta_{\bar g}u(t)- \e^{-2u(t)}\bar K + f -\alpha(t)\label{HG1}\\
              &=\e^{-2u(t)}\Delta_{\bar g}u(t) + \bar K \left(\frac1A-\e^{-2u(t)}\right)+f-\frac1A\int_Mf\e^{2u(t)} d\mu_{\bar g};\nonumber\\ 
u(0)&=u_0\in \mathcal{C}_{p,A},  \label{HG2}
\end{align}
where
\[\alpha(t)=\frac{1}{A}\left(\int_M fd\mu_{g(t)}-\bar K\right).\]
To find a solution of \eqref{HG1}, \eqref{HG2} on a short time interval, we consider the linear equation
\begin{align}
\partial_tu(t)&=\e^{-2v(t)}\Delta_{\bar g}u(t)+\bar K\left(\frac1A-\e^{-2v(t)}\right)+f-\frac1A\int_Mf \e^{2v(t)}d\mu_{\bar g}; \label{LineareHG1}\\
u(0)&=u_0\in\mathcal{C}_{p,A},\label{LineareHG2}
\end{align}
and use a fixed point argument in the Banach space
\begin{equation}
(X,\|\cdot\|_X):=(C([0,T]\times M),\|\cdot\|_{L^\infty([0,T]\times M)}).
\end{equation}
For this we first observe that equation \eqref{LineareHG1} is strongly parabolic for $v \in X$.
Furthermore, since $p>2$ and $M$ is compact, we have $u_0\in \mathcal{C}_{p,A} \subset H^2(M,\bar g)$, and therefore $u_0\in C(M)$. 

For the fixed point argument we fix $u_0\in \mathcal{C}_{p,A}$ and set
$$
R=R(u_0):=\|u_0\|_{L^\infty(M,\bar g)}+1.
$$
For fixed $T>0$ and $v\in X$, we then get, by \autoref{sec:appendix} in the appendix, a unique solution $u_v\in W^{2,1}_p((0,T)\times M)$ of \eqref{LineareHG1} which satisfies \eqref{LineareHG2} in the initial trace sense.
Here $W^{2,1}_p((0,T)\times M)$ denotes the space of functions $u \in L^p((0,T)\times M)$ which have weak derivatives $Du, D^2u$ and $\partial_t u$ in $L^p((0,T)\times M)$, so this space is compactly embedded in $C(X)$ by \autoref{W12p-embedding} in the appendix. 
On $X_R=\{ U\in X\mid \| U\|_X\le R\}$, we now define the function $\Phi$ as follows: for $v\in X_R$, let $\Phi(v)=:u_v$ be the unique solution of \eqref{LineareHG1}, \eqref{LineareHG2}.
First, we show that $\Phi:X_R\to X_R$ if $T>0$ is chosen small enough.

\begin{lemma}
\label{lemma-T-condition}  
If $T>0$ is fixed with 
  \begin{equation}
    \label{eq:T-condition}
    T \le \left( |\bar K|\e^{2(\|u_0\|_{L^\infty(M,\bar g)}+1)}+\|f\|_{L^\infty(M,\bar g)}\left(1+ \frac{\e^{2(\|u_0\|_{L^\infty(M,\bar g)}+1)}}{A}\right)    \right)^{-1}   
  \end{equation}
and $v\in X_R$, then $\Phi(v)\in X_R$.
\end{lemma}

\begin{proof}
With \autoref{max-principle} (ii) we directly get
\begin{equation}\label{BoundUv}
 \|\Phi(v)\|_{X} = \|u_v\|_{X} \le \|u_0^+\|_{L^\infty(M,\bar g)}+T d_T
\end{equation}
where
$$
d_T \le |\bar K|\e^{2\|v\|_{X}}+\|f\|_{L^\infty(M,\bar g)}+\frac{\|f\|_{L^\infty(M,\bar g)}\e^{2\|v\|_{X}}}{A} \le |\bar K|\e^{2R}+\|f\|_{L^\infty(M,\bar g)}\Bigl(1+ \frac{\e^{2R}}{A}\Bigr),
$$
hence
\begin{align*}
\|\Phi(v)\|_{X} & \le T\left(|\bar K|\e^{2R}+\|f\|_{L^\infty(M,\bar g)}\left(1+ \frac{\e^{2R}}{A}\right)\right)+\|u^+_0\|_{L^\infty(M,\bar g)}\\
&\le 1+ \|u_0\|_{L^\infty(M,\bar g)} = R,
\end{align*}
by \eqref{eq:T-condition} and since $R=\|u_0\|_{L^\infty(M,\bar g)}+1$, which shows the claim. 
\end{proof}

We now use Schauder's fixed point Theorem \cite{Schauder1930} to show the following proposition.
\begin{proposition}\label{existence}
  If $u_0 \in \mathcal{C}_{p,A}\subset W^{2,p}(M,\bar g)$ and $T>0$ is fixed with \eqref{eq:T-condition}, then there exists a short-time solution $u \in X \cap C^{\infty}((0,T)\times M)$ of \eqref{HG1}, \eqref{HG2}.\\
  Moreover, any such solution satisfies $u \in C([0,T), H^1(M,\bar g))$.   
\end{proposition}

\begin{proof}
{\bf Step 1:} First we recall Schauder's Theorem: If $H$ is a nonempty, convex, and closed subset of a Banach space $B$ and $F$ is a continuous mapping of $H$ into itself such that $F(H)$ is a relatively compact subset of $H$, then $F$ has a fixed point.

In our case, $B\hat{=}X=C([0,T] \times M)$, $H\hat{=}X_R=\{u\in X\mid \|u\|_X=\|u\|_{C_tC_x}\le R\}$, and $F\hat{=}\Phi$. So to show the existence of a fixed point of $\Phi$ in $X_R$, it remains to show that
\begin{enumerate}
\item $\Phi: X_R\to X_R$ is continuous and
\item $\Phi(X_R)\subset X_R$ is relatively compact.
\end{enumerate}

First, we show that $\Phi:X_R\to X_R$ is continuous. For this, let $v\in X_R$, and let $(v_n)_{n}\subset X_R$ be a sequence with $\|v_n-v\|_X\to0$. Moreover, let $u= \Phi(v)$ and $u_n= \Phi(v_n)$ for $n \in \N$.
By \autoref{sec:appendix}, we know that 
\[
  \|u_n\|_{W^{2,1}_p}\le C(\|u_0\|_{W^{2,p}(M,\bar g)}+\|d_n\|_{L^p_tL^p_x})\qquad \text{and}\qquad \|u\|_{W^{2,1}_p}\le C(\|u_0\|_{W^{2,p}(M,\bar g)}+\|d\|_{L^p_tL^p_x})\]
  for $n \in \N$ with
\[d_n(t):=\bar K\left(\frac1A-\e^{-2v_n(t)}\right)+f-\frac1A\int_Mf\e^{2v_n(t)}d\mu_{\bar g}\qquad \text{and}\qquad d(t):=\bar K\left(\frac1A-\e^{-2v(t)}\right)+f-\frac1A\int_Mf\e^{2v(t)}d\mu_{\bar g}.\]
Since $v_n\to v$ in $X$, we have $\e^{\pm2v_n}\to\e^{\pm2v}$ and therefore also $d_n\to d$ in $X$, which also implies that $d_n\to d$ in $L^p_tL^p_x$ for all $p$. Moreover, the difference $u_n-u= \Phi(v_n)-\Phi(v)$ fulfils the equation
\begin{align*}
\partial_t(u_n-u)(t)&=\e^{-2v_n(t)}\Delta_{\bar g}u_n(t)+d_n(t)-\e^{-2v(t)}\Delta_{\bar g}u(t)- d(t)\\
&=\e^{-2v_n(t)}\Delta_{\bar g}(u_n-u)(t)+(\e^{-2v_n(t)}-\e^{-2v(t)})\Delta_{\bar g}u(t)+d_n(t)-d(t).
\end{align*}
Since also $[u_n-u](0)=0$, we have, again by \autoref{sec:appendix},  
\begin{align*}
\|u_n-u\|_{W^{2,1}_p}&\le C\|(\e^{-2v_n}-\e^{-2v})\Delta_{\bar g}u+d_n-d\|_{L^p_tL^p_x}\\
&\le C\left(\|\e^{-2v_n}-\e^{-2v}\|_{X}\|\Delta_{\bar g}u\|_{L^p_tL^p_x}+\|d_n-d\|_{L^p_tL^p_x}\right)
\end{align*}
Since $\|\Delta_{\bar g}u\|_{L^p_tL^p_x}$ is finite, it thus follows that $\Phi(v_n)-\Phi(v)= u_n-u \to 0$ in $W^{2,1}_p$ and therefore also $\Phi(v_n)-\Phi(v) \to 0$ in $X$, since $W^{2,1}_p$ is embedded in $X$ by \autoref{W12p-embedding}.
Together with \ref{lemma-T-condition}, this shows the continuity of $\Phi:X_R\to X_R$.\

Next, we show that $\Phi(X_R)$ is relatively compact. For this let $(u_n)_{n\in\N}\subset \Phi(X_R)$ be an arbitrary sequence in $\Phi(X_R)$, and let $v_n\in X_R$ with $\Phi(v_n)=u_n$ for $n \in \N$.
So, by definition of $\Phi$ and by \autoref{sec:appendix}, we see that 
\begin{align*}
\|u_n\|_{W^{2,1}_p}&\le C
\left(\|u_0\|_{W^{2,p}(M,\bar g)}+\frac{T|\bar K|}{A}+\|\bar K\e^{-2v_n}\|_{L^p_tL^p_x}+\|f\|_{L^p_tL^p_x}+\left\|\frac1A\int_Mf\e^{2v_n}d\mu_{\bar g}\right\|_{L^p_tL^p_x}\right)\\
&\le C\left(\|u_0\|_{W^{2,p}(M,\bar g)}+\frac{T|\bar K|}{A}+|\bar K|\e^{2R}+T\|f\|_{L^\infty(M,\bar g)}+\frac{T}A\|f\|_{L^\infty(M,\bar g)}\e^{2R}\right)
\end{align*}
for $n \in \N$. Hence $(u_n)_{n\in\N}$ is uniformly bounded in $W^{2,1}_p((0,T)\times M)$. Using now that $W^{2,1}_p((0,T)\times M)$ is compactly embedded in $X$ by \autoref{W12p-embedding}, we conclude the claim.\\
We have thus proved that $\Phi$ has a fixed point $u$ in $X_R$, which then is a (strong) solution $u \in W^{2,1}_p((0,T)\times M)$ of \eqref{HG1}, \eqref{HG2}.

{\bf Step 2:} We now show that $u \in C^\infty((0,T)\times M)$. To see this, we first note the trivial fact that $u \in W^{2,1}_p((0,T)\times M)$ is a strong solution of \eqref{LineareHG1}, \eqref{LineareHG2} with $v = u$. Since then $v \in W^{2,1}_p((0,T)\times M) \subset C^\alpha([0,T]\times M)$, \cite[Theorems 5.9 and 5.10]{Lie96} imply the existence of a classical solution $\tilde u \in X \cap C^{2+\alpha',1+\alpha'}_{loc}((0,T)\times M)$ of \eqref{LineareHG1}, \eqref{LineareHG2} with $v = u$ for some $\alpha'>0$. Here $C^{2+\alpha',1+\alpha'}_{loc}((0,T)\times M)$ denotes the space of functions $f \in C^{2,1}((0,T)\times M)$ with the property that $\partial_t f$ and all derivatives up to second order of $f$ with respect to $x \in M$ are locally $\alpha'$-H\"older continuous. 
In particular, $\tilde u \in W^{2,1}_p((\varepsilon,T-\varepsilon)\times M)$ for $\varepsilon \in (0,T)$. The function $w:= u- \tilde u \in W^{2,1}_p((\varepsilon,T-\varepsilon)\times M)$
is then a strong solution of the initial value problem
$$
\partial_t w(t)=\e^{-2v(t)}\Delta_{\bar g}w(t) \quad \text{for $t \in (\varepsilon,T-\varepsilon)$}, \qquad w(\varepsilon)= u(\varepsilon,\cdot)-\tilde u(\varepsilon,\cdot).
$$
By \autoref{max-principle} (ii) we then have $|w| \le \|u(\varepsilon,\cdot)-\tilde u(\varepsilon,\cdot)\|_{L^\infty(M,\bar g)}$ on $(\varepsilon,T-\varepsilon)\times M$, whereas $\|u(\varepsilon,\cdot)-\tilde u(\varepsilon,\cdot)\|_{L^\infty(M,\bar g)} \to 0$ as $\varepsilon \to 0$ by the continuity of $u$ and $\tilde u$. It thus follows that
$u \equiv \tilde u$ on $(0,T)\times M)$, and therefore $u \in C^{2+\alpha',1+\alpha'}_{loc}((0,T)\times M)$. Since $u$ solves \eqref{LineareHG1}, \eqref{LineareHG2} with $v = u \in C^{2+\alpha',1+\alpha'}_{loc}((0,T)\times M)$, we can apply \cite[Theorems 5.9]{Lie96} and the above argument again to get $u \in 
C^{4+\alpha'',2+\alpha''}_{loc}((0,T)\times M)$ for some $\alpha''>0$.
Repeating this argument inductively, we get $u \in C^{k,\frac{k}{2}}_{loc}((0,T)\times M)$ for every $k>0$, and hence $u \in C^\infty((0,T)\times M)$.\\
{\bf Step 3:} It remains to show that any solution $u \in X \cap C^{\infty}((0,T)\times M)$ of \eqref{HG1}, \eqref{HG2} also satisfies $u \in C([0,T), H^1(M,\bar g))$. Since $u \in C^{\infty}((0,T)\times M)$, only the continuity in $t=0$ needs to be proved. Setting $\phi(t)= \|u(t)\|_{H^1(M,\bar g)}^2$ for $t \in (0,T)$, we see that
\begin{align*}
\frac{1}{2}(\phi(t_2)-\phi(t_1)) &=  \frac{1}{2} \int_{t_1}^{t_2}\partial_t \|u(t)\|_{H^1(M,\bar g)}^2\,dt =\int_{t_1}^{t_2}\int_{M} \Bigl(u(t) \partial_t u(t) + \nabla u(t) \nabla \partial_t u(t)\Bigr)d\mu_{\bar g}dt \\
&=\int_{t_1}^{t_2}\int_{M} \Bigl(u(t) \partial_t u(t) - [\Delta u(t)] \partial_t u(t)\Bigr)d\mu_{\bar g}dt
\end{align*}
and therefore, by H\"older's inequality,
\begin{align*}
\frac{1}{2}|\phi(t_2)-\phi(t_1)|&\le  \int_{t_1}^{t_2} \int_{M}\bigl(|u||\partial_t u|+ |\Delta u||\partial_t u|\bigr)d\mu_{\bar g}dt\\
&\le C \|\partial_t u\|_{L^p((0,T) \times M)} \bigl(\|u\|_{L^p((0,T) \times M)} + \|\Delta u\|_{L^p((0,T) \times M)} \bigr) (t_2-t_1)^\beta\\
&\le C \|u\|_{W^{1,2}_p((0,T) \times M)}(t_2-t_1)^\beta,
\end{align*}
for $0<t_1<t_2<T$ with some $\beta >0$ depending on $p>2$, which implies that the function $\phi$ is uniformly continuous and therefore bounded on $(0,T)$.

We now assume by contradiction that $u$ is not continuous at $t=0$ with respect to the $H^1(M,\bar g)$-norm. Then there exists a sequence $(t_n)_{n\in\N}$ in $(0,T)$ and $\varepsilon>0$ with
$t_n \to 0^+$ as $n \to \infty$ and
\begin{equation}
   \label{eq:contradiction-inequality}
\|u(t_n)-u_0\|_{{H^1(M,\bar g)}} \ge \varepsilon \qquad \text{for all $n \in \N$.}   
\end{equation}
Since $\|u(t_n)\|_{H^1(M,\bar g)}^2 = \phi(t_n)$ remains bounded as $n \to \infty$, we conclude that, passing to a subsequence, the sequence $u(t_n)$ converges weakly in $H^1(M,\bar g)$ and therefore strongly in $L^2(M, \bar g)$. Since the strong $L^2$-limit of $u(t_n)$ must be $u_0=u(0)$ as a consequence of the fact that $u \in X$, we deduce that $u(t_n) \rightharpoonup u_0$ weakly in ${H^1(M,\bar g)}$ as $n \to \infty$. Combining this information with \autoref{sec:appendix} from the appendix, we deduce that
\begin{equation}
\label{eq:limsup-liminf-ineq}
\limsup_{n \to \infty}\|u(t_n)\|_{H^1(M,\bar g)}^2 \le \|u_0\|_{H^1(M,\bar g)}^2 \le \liminf_{n \to \infty}\|u(t_n)\|_{H^1(M,\bar g)}^2   
\end{equation}
and therefore $\|u(t_n)\|_{H^1(M,\bar g)} \to \|u_0\|_{H^1(M,\bar g)}$. Note here that this part of \autoref{sec:appendix} applies since $u$ solves \eqref{LineareHG1}, \eqref{LineareHG2} with $v = u  \in W^{2,1}_p((0,T)\times M) \subset C^\alpha([0,T]\times M)$ for some $\alpha>0$. 
From \eqref{eq:limsup-liminf-ineq} and the uniform convexity of the Hilbert space ${H^1(M,\bar g)}$, we conclude that $u(t_n) \to u_0$ strongly in $H^1(M,\bar g)$, contrary to \eqref{eq:contradiction-inequality}.
\end{proof}

\subsection{Uniqueness}\label{Uniqueness}
We now show that the solution from \autoref{existence} is unique.
\begin{lemma}
\label{lemma-uniqueness-nonlinear-short-time}
Let $u_0 \in W^{2,p}(M,\bar g)$, $p>2$, and $T>0$ be fixed with \eqref{eq:T-condition}. Then the short-time solution of $u \in X \cap C^{\infty}((0,T)\times M)$ of \eqref{HG1}, \eqref{HG2} given by \autoref{existence} is unique.
\end{lemma}

\begin{proof}
Let $u_1, u_2 \in  X \cap C^{\infty}((0,T)\times M)$ be two solutions of \eqref{HG1}, \eqref{HG2}. The difference $u:= u_1-u_2 \in  X \cap C^{\infty}((0,T)\times M)$ then fulfils 
\begin{equation}\label{diff-equation-uniqueness}
\begin{split}
\partial_t u(t)&=\e^{-2u_1(t)}\Delta_{\bar g}u_1(t)-\e^{-2u_2(t)}\Delta_{\bar g}u_2(t)\\
               &\phantom{aaaaa}-\bar K(\e^{-2u_1(t)}-\e^{-2u_2(t)})-\frac1A\int_Mf(\e^{2u_1(t)}-\e^{2u_2(t)})d\mu_{\bar g} \\
&=  \e^{-2u_1(t)}\Delta_{\bar g}u(t) + \Delta_{\bar g}u_2(t)\bigl( \e^{-2u_1(t)} -\e^{-2u_2(t)}\bigr)\\
               &\phantom{aaaaa}-\bar K(\e^{-2u_1(t)}-\e^{-2u_2(t)})-\frac1A\int_Mf(\e^{2u_1(t)}-\e^{2u_2(t)})d\mu_{\bar g}\quad \text{for $t \in (0,T)$} .
\end{split}
\end{equation}

In the following, the letter $C$ denotes different positive constants. Multiplying \eqref{diff-equation-uniqueness} with $2u$ and integrating over $M$ gives
\begin{align}
 \frac{d}{dt}& \|u(t)\|_{L^2(M,\bar g)}^2 = 2 \int_{M} u(t) \partial_t u(t) d\mu_{\bar g} \nonumber\\
 &=2\int_M\e^{-2u_1(t)}u(t)\Delta_{\bar g}u(t)d\mu_{\bar g} + 2\int_Mu(t)\Delta_{\bar g}u_2(t)\bigl( \e^{-2u_1(t)} -\e^{-2u_2(t)}\bigr)d\mu_{\bar g}\\
               &\qquad-2\int_M\bar Ku(t)(\e^{-2u_1(t)}-\e^{-2u_2(t)})d\mu_{\bar g}-\frac2A\int_Mf(\e^{2u_1(t)}-\e^{2u_2(t)})d\mu_{\bar g}\int_Mu(t)d\mu_{\bar g}\nonumber\\
&\le 2 \int_M \e^{-2u_1(t)} u(t) \Delta_{\bar g}u(t) + 2\int_M V(t,x)u^2(t) + 2\rho(t) \|u(t)\|_{L^2(M,\bar g)} \int_M |u(t)|d\mu_{\bar g} \nonumber\\
 &\le 2 \left(-\int_M \e^{-2u_1(t)} |\nabla_{\bar g} u(t)|_{\bar g}^2 + 2 \int_M  \e^{-2u_1(t)} u(t) \langle\nabla_{\bar g} u_1(t),\nabla_{\bar g}u(t)\rangle_{\bar g}d\mu_{\bar g}\right)\nonumber\\
  &\qquad + 2\|V(t,\cdot)\|_{L^p(M,\bar g)} \|u(t)\|_{L^{2p'}(M,\bar g)}^2 + C \|u(t)\|_{L^2(M,\bar g)}^2 \nonumber\\
&\le C \|\nabla_{\bar g}u_1(t)\|_{L^4(M,\bar g)} \|u(t)\|_{L^4(M,\bar g)} \|\nabla_{\bar g}u(t)\|_{L^2(M,\bar g)}\nonumber\\
&\qquad+ 2\|V(t,\cdot)\|_{L^p(M,\bar g)} \|u(t)\|_{L^{2p'}(M,\bar g)}^2 +C \|u(t)\|_{L^2(M,\bar g)}^2 \nonumber\\  
&\le C \Bigl(\|u_1(t)\|_{H^2(M,\bar g)} \|u(t)\|_{H^1(M,\bar g)}^2  + 2\|V(t,\cdot)\|_{L^p(M,\bar g)} \|u(t)\|_{H^1(M,\bar g)}^2 +  \|u(t)\|_{L^2(M,\bar g)}^2\Bigr) \nonumber\\
  &\le C \Bigl(\|u_1(t)\|_{H^2(M,\bar g)}   + 2\|V(t,\cdot)\|_{L^p(M,\bar g)}+1\Bigr) \|u\|_{H^1(M,\bar g)}^2,\label{gronwall-H1-1}
\end{align}
with functions $V \in L^p((0,T) \times M) \cap C^\infty((0,T) \times M)$ and $\rho \in L^\infty(0,T)$. 
Here we used the Sobolev embeddings $H^1(M,\bar g) \hookrightarrow L^\rho(M)$ for $\rho \in [1,\infty)$. Multiplying \eqref{diff-equation-uniqueness} with $-2\Delta u$ and integrating over $M$ yields
\begin{align}
  \frac{d}{dt} &\|\nabla_g u(t)\|_{L^2(M,\bar g)}^2 = 2 \int_M \nabla  u(t) \nabla \partial_t u(t)d\mu_{\bar g}= -2 \int_M \Delta_g u(t) \partial_t u(t)d\mu_{\bar g} \nonumber\\
  &\le - 2 \int_M \e^{-2u_1(t)} |\Delta_{\bar g}u(t)|^2 d\mu_{\bar g} +2 \int_M V(t,x)|u(t)||\Delta u(t)|d\mu_{\bar g}  \nonumber\\
                                                 &\le -\kappa  \|\Delta_{\bar g}u(t)\|_{L^2(M,\bar g)}^2 +
                                                   2\|V(t,\cdot)\|_{L^p(M,\bar g)} \|u\|_{L^\alpha(M,\bar g)} \|\Delta_g u\|_{L^2(M,\bar g)} \nonumber\\
                                                &\le -\kappa  \|\Delta_{\bar g}u(t)\|_{L^2(M,\bar g)}^2 + 
                                                  \frac{1}{\kappa} \|V(t,\cdot)\|_{L^p(M,\bar g)}^2 \|u\|_{L^\alpha(M,\bar g)}^2 + \kappa \|\Delta_g u\|_{L^2(M,\bar g)}^2 \nonumber\\
&=  \frac{1}{\kappa} \|V(t,\cdot)\|_{L^p(M,\bar g)}^2 \|u\|_{L^\alpha(M,\bar g)}^2 \le C \|V(t,\cdot)\|_{L^p(M,\bar g)}^2 \|u\|_{H^1(M,\bar g)}^2,  \label{gronwall-H1-2}
\end{align}
where we used first H\"older's inequality with $\alpha = \frac{2p}{p-2}$, then Young's inequality and finally Sobolev embeddings again. Here we note that, by making $C>0$ larger if necessary, we may assume that the constants are the same in \eqref{gronwall-H1-1} and \eqref{gronwall-H1-2}. Combining these estimates gives
\begin{equation}
  \label{eq:gronwall-condition}
  \frac{d}{dt} \|u(t)\|_{H^1(M,\bar g)}^2 \le g(t) \|u(t)\|_{H^1(M,\bar g)}^2 \qquad \text{for $t \in (0,T)$}
\end{equation}
with the function $g \in L^1(0,T)$ given by $g_1(t)= C \Bigl(\|u_1(t)\|_{H^2(M,\bar g)}   + 3\|V(t,\cdot)\|_{L^p(M,\bar g)}+1\Bigr)$. Integrating and using the fact that $u \in C([0,T), H^1(M,\bar g))$ by \autoref{existence} with $u(0)=u_1(0)-u_2(0)=0$, we see that 
$$
\|u(t)\|_{H^1(M,\bar g)}^2 \le \int_0^t g(s)\|u(s)\|_{H^1(M,\bar g)}^2\,ds \qquad \text{for $t \in [0,T)$.}
$$
It then follows from Gronwall's inequality \cite{CazHar99} that $\|u(t)\|_{H^1(M,\bar g)}^2 \equiv 0$ on $[0,T)$, hence $u_1 \equiv u_2$.
\end{proof}

\subsection{Global Existence}

Let $f \in C^\infty(M)$, $A>0$, $p>2$ and $u_0  \in \mathcal{C}_{p,A}$. In this section, we wish to show that the (unique) local solution 
\[u\in C([0,T]\times M)\cap C([0,T],H^1(M,\bar g))\cap C^\infty((0,T)\times M)\]
of the initial value problem \eqref{HG1}, \eqref{HG2} for small $T>0$ can be extended to a global und uniformly bounded solution defined for all positive times.

We first need the following local boundedness property on open time intervals.

\begin{lemma}\label{UpperULowerK}
Let, for some $T>0$, $u\in C([0,T)\times M)\cap C([0,T),H^1(M,\bar g))\cap C^\infty((0,T)\times M)$
be a solution of \eqref{HG1}, \eqref{HG2} on $[0,T)$. Then we have 
\begin{equation}
  \label{eq:UpperULowerL}
  \sup_{t\in[0,T)}\|u(t)\|_{L^\infty(M,\bar g)} \le \mathcal{M}
\end{equation}
with some $\mathcal{M}=\mathcal{M}(\|u_0\|_{L^\infty(M,\bar g)},\|f\|_{L^\infty(M,\bar g)},T)>0$ which is increasing in all of its variables.
\end{lemma}

\begin{proof}
Since $\bar K <0$, we have 
$$
\partial_tu(t)=\e^{-2u(t)}\Delta_{\bar g}u(t)-\e^{-2u(t)}\bar K+f-\alpha(t)=\e^{-2u(t)}\Delta_{\bar g}u(t)+\e^{-2u(t)}|\bar K|+ f-\alpha(t) \qquad \text{for $t \in [0,T)$}
$$
by \eqref{HG1}, where
$$
|\alpha(t)| \le \alpha_0:=\|f\|_{L^\infty(M,\bar g)}+\frac{|\bar K|}{A} \qquad \text{for $t \in [0,T)$}
$$
by \eqref{Lower-Upper-Boundtildealpha}. Hence the function $v=-u$ satisfies  
$$
\partial_t v(t)= \e^{2 v(t)}\Delta_{\bar g}v(t)-\e^{2v(t)}|\bar K| -f+\alpha(t) \le
\e^{2 v(t)}\Delta_{\bar g}v(t)+ c \qquad \text{for $t \in (0,T)$} 
$$
with $c = \|f\|_{L^\infty(M,\bar g)}+ \alpha_0$. Next, let $(T_k)_k \subset (0,T)$ be a sequence with $T_k\to T$ for $k\to\infty$. For fixed $k \in \N$ the continuous function $\e^{2v}$ is then bounded from below by a positive constant on the compact set $[0,T_k]\times M$. Therefore \autoref{max-principle} (ii) from the appendix implies that 
$$
v(t,x) \le \|u_0\|_{L^\infty(M,\bar g)} + T_k c \qquad \text{for $(t,x) \in [0,T_k]\times M$.}
$$
Letting $k \to \infty$, we deduce that 
\begin{equation}
  \label{eq:lower-bound-U}
u(t,x) =- v(t,x) \ge - \|u_0\|_{L^\infty(M,\bar g)} - T c \qquad \text{for $(t,x) \in [0,T)\times M$.}
\end{equation}
In order to derive an upper bound for $u$, we now observe that
$$
\partial_tu(t)=\e^{-2u(t)}\Delta_{\bar g}u(t)+\e^{-2u(t)}|\bar K|+ f-\alpha(t)
\le \e^{-2u(t)} \Delta_{\bar g}u(t) + \e^{2(\|u_0\|_{L^\infty(M,\bar g)} + T c)}+ c 
$$
on $M$ for $t \in [0,T)$. Applying \autoref{max-principle} (ii) in the same way as above therefore gives  
\begin{equation}
  \label{eq:upper-bound-U}
u(t,x) \le \|u_0\|_{L^\infty(M,\bar g)} +  T \Bigl(\e^{2(\|u_0\|_{L^\infty(M,\bar g)} + T c)}+ c\Bigr).
\end{equation}
Combining \eqref{eq:lower-bound-U} and \eqref{eq:upper-bound-U} yields
\begin{equation}\label{UniformBoundU}
\begin{split}
  &\sup_{\substack{t\in[0,T)\\ x\in M}}|u(t,x)| \le \mathcal{M}\quad \text{with}\\
  &\mathcal{M}=\mathcal{M}(\|u_0\|_{L^\infty}(M,\bar g),\|f\|_{L^\infty(M,\bar g)},T):= \|u_0\|_{L^\infty} +  T \Bigl(\e^{2(\|u_0\|_{L^\infty(M,\bar g)} + T c)}+ c\Bigr),
  \end{split}
\end{equation}
as claimed in \eqref{eq:upper-bound-U}.
\end{proof}

\begin{corollary}
The initial value problem \eqref{HG1}, \eqref{HG2} admits a unique global solution $u\in C([0,\infty)\times M)\cap C([0,\infty),H^1(M,\bar g))\cap C^\infty((0,\infty)\times M)$.
\end{corollary}

\begin{proof}
This follows from \autoref{existence}, \autoref{lemma-uniqueness-nonlinear-short-time} and \autoref{UpperULowerK} by a standard continuation argument using condition ~\eqref{eq:T-condition}.
\end{proof}

In the next lemma, with the help of \eqref{APrioriBound}, we turn \eqref{eq:UpperULowerL} into a uniform estimate for all time.

\begin{lemma}\label{Linfty}
  Let $u$ be the global, smooth solution of the initial value problem \eqref{HG1}, \eqref{HG2}. Then we have
  $$
  \sup_{t>0}\|u(t)\|_{L^\infty(M,\bar g)}\le \mathcal{N}
   $$
with some $\mathcal{N}=\mathcal{N}(u_0,\|f\|_{L^\infty(M,\bar g)})>0$ which is increasing in its second variable.
\end{lemma}

\begin{proof}
We argue similarly as in the proof of \cite[Lemma 2.5]{Str20}.

By using the fact that $u(t)$ is a volume preserving solution of \eqref{HG1} with $u(0)=u_0\in \mathcal{C}_{p,A}$ and therefore $\int_M\e^{2u(t)}d\mu_{\bar g}\equiv A$, we get with \eqref{Jensen-consequence} and the fact that $\bar K<0$ that
\begin{equation}\label{LowerBoundEnergy}
\begin{split}
E_{f}(u(t))&=\frac12\|\nabla_{\bar g}u(t)\|^2_{L^2(M, \bar g)}+\int_M\bar Ku(t)d\mu_{\bar g}-\frac12\int_Mf\e^{2u(t)}d\mu_{\bar g}\\
&\ge\frac{\bar K}{2}\int_M2u(t)d\mu_{\bar g}-\frac12\int_Mf\e^{2u(t)}d\mu_{\bar g} \ge\frac{\bar K}{2}\log(A)-\frac{A}{2}\|f\|_{L^\infty(M,\bar g)}>-\infty.
\end{split}
\end{equation}
For the function
\begin{equation}
  \label{eq:def-F}
t \mapsto  F(t):=\int_M|\partial_tu(t)|^2d\mu_{g(t)}=\int_M|\partial_tu(t)|^2\e^{2u(t)}d\mu_{\bar g},
\end{equation}
we then obtain, by combining \eqref{LowerBoundEnergy} with \eqref{APrioriBound}, the estimate
  \begin{equation}
\label{UniformEstimate}
    \int_0^\infty F(t)dt=\lim_{T \to \infty} \int_0^T\int_M|\partial_tu(t)|^2d\mu_{g(t)}dt\le E_{f}(u_0)+\frac{|\bar K|}{2}|\log(A)|+\frac{A}{2}\|f\|_{L^\infty(M,\bar g)}.
\end{equation}
Hence, for any $T>0$ we find $t_T\in[T,T+1]$ such that
\begin{align}
  F(t_T)=\inf_{t\in(T,T+1)}F(t)&\le E_{f}(u_0)+\frac{|\bar K|}{2}|\log(A)|+\frac{A}{2}\|f\|_{L^\infty(M,\bar g)} \nonumber\\
                               &\le \frac12 \|\nabla u_0\|_{L^2(M,\bar g)}^2 + |\bar K| \bigl(\frac{1}{2}|\log(A)|+ \|u\|_{L^1(M,\bar g)}\bigr) + A \|f\|_{L^\infty(M,\bar g)} \nonumber\\
  &= d_1 + d_2 \|f\|_{L^\infty(M,\bar g)} \label{EstimateFtT}
\end{align}
with constants $d_i = d_i(u_0)>0$. Here we used \eqref{eq:C-A-energy-ineq}.

So, at time $t_T$ we get with \eqref{PCFVol1}, H\"olders inequality, Young's inequality, \eqref{Uniformexp}, and \eqref{EstimateFtT} that
\begin{equation}\label{DeltaUniform}
\begin{split}
\|\Delta_{\bar g}&u(t_T)\|_{L^{\frac32}(M,\bar g)}\\
&\le\|\e^{2u(t_T)}\partial_tu(t_T)\|_{L^{\frac32}(M,\bar g)}+\|\bar K\|_{L^{\frac32}(M,\bar g)}+\|\e^{2u(t_T)}f\|_{L^{\frac32}(M,\bar g)}+\|\e^{2u(t_T)}\alpha(t_T)\|_{L^{\frac32}(M,\bar g)}\\
&\le\|\e^{u(t_T)}\|_{L^6(M,\bar g)}F(t_T)^{\frac12}+|\bar K|+
\|f\|_{L^\infty(M,\bar g)} \Bigl(\int_M\e^{3u(t_T)}d\mu_{\bar g}\Bigr)^{\frac23}+ |\alpha(t_T)| \Bigl(\int_M\e^{3u(t_T)} d\mu_{\bar g}\Bigr)^{\frac23}\\
&\le \frac12 \|\e^{u(t_T)}\|_{L^6(M,\bar g)}^2  + \frac12  F(t_T) +|\bar K|+ \frac{1}{3}\Bigl( \|f\|_{L^\infty(M,\bar g)}^3 + |\alpha(t_T)|^3\Bigr)+\frac{4}{3} \int_M\e^{3u(t_T)} d\mu_{\bar g}\\
&\le \frac12 \Bigl(\nu_0(u_0,6) \e^{\nu_1(u_0,6) \|f\|_{L^\infty(M,\bar g)}}\Bigr)^{1/3}
+ \frac12 \Bigl(d_1 + d_2 \|f\|_{L^\infty(M,\bar g)}\Bigr)  +|\bar K|\\
&+ \frac{1}{3}\Bigl( \|f\|_{L^\infty(M,\bar g)}^3 + |\alpha(t_T)|^3\Bigr)+\frac{4}{3} \nu_0(u_0,3) \e^{\nu_1(u_0,3) \|f\|_{L^\infty(M,\bar g)}}\\
&\le d_3  \e^{d_4 \|f\|_{L^\infty(M,\bar g)}}
\end{split}
\end{equation}
with constants $d_i=d_i(u_0)$, $i\in\{3,4\}$. Here the constants $\nu_i(u_0,3)$, $i\in\{0,1\}$ are given in \eqref{Uniformexp}.

Furthermore, with Sobolev's embedding theorem we have $W^{2,\frac{3}{2}}(M)\subset C^{0,\frac23} \subset L^\infty(M,\bar g)$. Therefore we get with Poincar\'e's inequality, the Calder\'on--Zygmund inequality for closed surfaces, and with
\eqref{DeltaUniform} that
\begin{equation}
\begin{split}
\|u(t_T)-\bar u(t_T)\|^{\frac32}_{L^\infty(M,\bar g)}&\le d_5\|u(t_T)-\bar u(t_T)\|^{\frac32}_{W^{2,\frac32}(M,\bar g)}\le d_6 \|\nabla_{\bar g}^2u(t_T)\|^{\frac32}_{L^{\frac32}(M,\bar g)}\\
&\le d_7 \|\Delta_{\bar g}u(t_T)\|^{\frac32}_{L^{\frac32}(M,\bar g)}\le d_8 \e^{d_9 \|f\|_{L^\infty(M,\bar g)}}.
\end{split}
\end{equation}
with constants $d_i>0$, $i\in\{5,6,7\}$ and $d_i=d_i(u_0)>0$, $i\in\{8,9\}$. With \eqref{UniBarU} we therefore obtain the uniform bound
\begin{equation}
\|u(t_T)\|_{L^\infty(M,\bar g)}\le d_8 \e^{d_9 \|f\|_{L^\infty(M,\bar g)}} +\max\left\{|m_0|+ m_1 \|f\|_{L^{\infty}(M)},\frac12|\log(A)|\right\}.
\end{equation}
Upon shifting time by $t_T$, we therefore get from \autoref{UpperULowerK} 
\begin{equation}
\begin{split}
\sup_{s\in[T+1,T+2]}&\|u(s)\|_{L^\infty(M,\bar g)} \le \sup_{s\in[t_T,t_T+3)}\|u(s)\|_{L^\infty(M,\bar g)}
\le \mathcal{M}(\|u(t_T)\|_{L^\infty(M,\bar g)},\|f\|_{L^\infty(M,\bar g)},3)\\
&\le \mathcal{M}\Bigl(d_8 \e^{d_9 \|f\|_{L^\infty(M,\bar g)}} +\max\left\{|m_0|+ m_1 \|f\|_{L^{\infty}(M,\bar g)},\frac12|\log(A)|\right\},\|f\|_{L^\infty(M,\bar g)},3\Bigr)\\
&=:\mathcal{N}(u_0,\|f\|_{L^\infty(M,\bar g)}).
\end{split}
\end{equation}
Since $\mathcal{M}$ is increasing in its first and second variables by \autoref{UpperULowerK}, we see that $\mathcal{N}$ is increasing in $\|f\|_{L^\infty(M,\bar g)}$, as claimed.
Since $T>0$ was arbitrary, the claim follows.
\end{proof}

\subsection{Convergence of the Flow}\label{SectionConvergenceFlow}
Let $f \in C^\infty(M)$, $A>0$, $p>2$ and $u_0  \in \mathcal{C}_{p,A}$ as before, and let $u$ denote the global, smooth solution of the initial value problem \eqref{HG1}, \eqref{HG2}. In this section we shall show that 
for a suitable sequence $t_l\to\infty$, $l\to\infty$, the associated sequence of metrics $g(t_l)$ tends to a limit metric $g_\infty=\e^{2u_\infty}\bar g$ with Gauss curvature $K_{g_\infty}$, which then implies that $K_{g_\infty}=f-\alpha^\infty$ with a constant $\alpha^\infty$. Afterwards, we shall have a closer look at this constant $\alpha^\infty$.

By
\eqref{UniformEstimate}, we know that, for a suitable sequence $t_l\to\infty$, $l\to\infty$ we have
\begin{equation}\label{Conv1}
\int_M|\partial_tu(t_l)|^2d\mu_{g(t_l)}=\int_M|f-K_{g_l}-\alpha(t_l)|^2d\mu_{g(t_l)}\to0\quad\text{for }l\to\infty.
\end{equation}

We can strengthen this observation as follows. 
\begin{lemma}\label{KonvergenzF}
For $F(t)=\int_M|\partial_tu(t)|^2d\mu_{g(t)}$ as above, we have $F(t)\to0$ for $t\to\infty$.
\end{lemma}

\begin{proof}
First we consider the evolution equation of the curvature $K_{g(t)}$ and of $\alpha(t)$. By the Gauss equation \eqref{GaussEquation} and \eqref{HG1} we have
\begin{equation}\label{EvolutionCurvature}
\begin{split}
\partial_tK_{g(t)}&=\partial_t(-\e^{-2u(t)}\Delta_{\bar g}u(t)+\e^{-2u(t)}\bar K)\\
&=-2\partial_tu(t)K_{g(t)}-\Delta_{g(t)}\partial_tu(t)\\
&=2K_{g(t)}(K_{g(t)}-f+\alpha(t))+\Delta_{g(t)}(K_{g(t)}-f+\alpha(t))\\
&=2(K_{g(t)}-f+\alpha(t))^2+2(f-\alpha(t))(K_{g(t)}-f+\alpha(t))+\Delta_{g(t)}(K_{g(t)}-f+\alpha(t))\\
&=2(\partial_t u(t))^2-2(f-\alpha(t))\partial_t u(t)- \Delta_{g(t)}\partial_t u(t)
\end{split}
\end{equation}
for $t>0$. Moreover, by \eqref{eq:definition-alpha} we have 
\begin{equation}\label{EvolutionBeta}
\frac{d}{dt}\alpha(t)=\frac2A\int_Mf\e^{2u(t)}\partial_tu(t)d\mu_{\bar g}=\frac2A\int_Mf \partial_tu(t)d\mu_{g(t)}.
\end{equation}
Combining \eqref{HG1}, \eqref{EvolutionCurvature} and \eqref{EvolutionBeta}, we arrive at
\begin{equation}\label{Combined}
\begin{split}
&\partial_{tt}u(t) = \partial_t \bigl(f - K_{g(t)}-\alpha(t)\bigr)\\
&=-2(\partial_t u(t))^2+2(f-\alpha(t))\partial_t u(t)+ \Delta_{g(t)}\partial_t u(t)+\frac2A\int_Mf \partial_tu(t)d\mu_{g(t)}.
\end{split}
\end{equation}
We therefore get, using \eqref{Lower-Upper-Boundtildealpha}, that
\begin{equation}\label{EvolutionF1}
\begin{split}
  \frac12&\frac{d}{dt}F(t)= \frac12 \frac{d}{dt}\int_{M}|\partial_t u(t)|^2 \e^{2u(t)}d \mu_{\bar g}=  \int_{M}\bigl(\partial_t u(t) \partial_{tt}u(t) + |\partial_t u(t)|^2\partial_t u(t) \bigr)d \mu_{g(t)} \\
  &=\int_{M}\Bigl(- (\partial_t u(t))^3 + 2\bigl(f-\alpha(t)\bigr)(\partial_t u(t))^2
+ \partial_t u(t) \Delta_{g(t)}\partial_t u(t)\Bigr) d \mu_{g(t)} \\  
&\le - \int_{M} (\partial_t u(t))^3d \mu_{g(t)}  + 2\bigl(\|f\|_{L^\infty(M,\bar g)}+\alpha_0\bigr)F(t)-G(t)
\end{split}
\end{equation}
with
\[
  G(t):=\int_M |\nabla_{g(t)}\partial_tu(t)|^2_{g(t)}d\mu_{g(t)} \qquad \text{for $t >0$.}
  \]
With \autoref{eindSob}, applied with $r=3$, $C_{\text{GNL}}= C_{\text{GNL}}(3)>0$, \eqref{HG1} and \autoref{Linfty} we can furthermore estimate
\begin{equation}\label{Estimate2}
\begin{split}
  - \int_{M} &(\partial_t u(t))^3d \mu_{g(t)} \le \int_{M} |\partial_tu(t)|^3 \e^{2u(t)} d\mu_{\bar g} 
  \le \e^{2\mathcal{N}} \|\partial_tu(t)\|^3_{L^3(M,\bar g)}\\
&\le   \e^{2\mathcal{N}} C_{\text{GNL}}
  \|\partial_tu(t)\|_{L^2(M,\bar g)}^2 \|\partial_tu(t)\|_{H^1(M,\bar g)}\\
  &=   \e^{2\mathcal{N}} C_{\text{GNL}} \int_M|\partial_tu(t)|^2\e^{-2u(t)}d\mu_{g(t)} 
  \left(\int_M|\partial_tu(t)|^2\e^{-2u(t)}d\mu_{g(t)} + \int_M|\nabla_{g(t)} \partial_tu(t)|^2d\mu_{g(t)} \right)^{\frac12}\\
&\le  \e^{6\mathcal{N}} C_{\text{GNL}} \int_M|\partial_tu(t)|^2 d\mu_{g(t)} 
  \left(\int_M|\partial_tu(t)|^2d\mu_{g(t)} + \int_M|\nabla_{g(t)} \partial_tu(t)|^2d\mu_{g(t)} \right)^{\frac12}\\
&=  \e^{6\mathcal{N}} C_{\text{GNL}} F(t)
\Bigl(F(t) + G(t) \Bigr)^{\frac12} \le \frac{\bigl(\e^{6\mathcal{N}} C_{\text{GNL}}\bigr)^2}{2}F^2(t) + \frac{1}{2}
\Bigl(F(t) + G(t) \Bigr),
\end{split}
\end{equation}
where we used Young's inequality and the fact that
\[
  G(t) =\int_M |\nabla_{g(t)}\partial_tu(t)|^2_{g(t)}d\mu_{g(t)}=
  \int_M|\nabla_{\bar g}\partial_tu(t)|^2_{\bar g}d\mu_{\bar g} \qquad \text{for $t >0$.}
  \]
Combining \eqref{EvolutionF1} and \eqref{Estimate2} and using that $G(t) \ge 0$ gives 
\begin{equation}\label{EvolutionF2}
\begin{split}
  \frac{d}{dt}F(t) \le \frac{d}{dt}F(t)+G(t)&\le \bigl(\e^{6\mathcal{N}} C_{\text{GNL}}\bigr)^2 F^2(t)+ \bigl(4(\|f\|_{L^\infty(M,\bar g)}+\alpha_0)+1\bigr) F(t)\\
&=:\tilde C_1F(t)+\tilde C_2 F^2(t).
\end{split}
\end{equation}
By integrating \eqref{EvolutionF2} over $(t_l,t)\subset(t_l,T)$ and taking the supremum over $t \in (t_l,T)$ we get\begin{align*}
\sup_{t\in(t_l,T)}F(t)&\le F(t_l)+\tilde C_1\int_{t_l}^TF(t)dt+\tilde C_2\int_{t_l}^TF^2(t)dt\\
&\le F(t_l)+\tilde C_1\int_{t_l}^\infty F(t)dt+ \tilde C_2\sup_{t\in(t_l,T)}F(t)\int_{t_l}^\infty F(t)dt
\end{align*}
With \eqref{UniformEstimate} we also have $\int_{t_l}^\infty F(t)dt\to0$ for $l\to\infty$ and thus $1-\tilde C_2\int_{t_l}^\infty F(t) dt>0$ for $l$ sufficiently large. For these $l$ and $T > t_l$ we thus have 
\[\sup_{t\in(t_l,T)}F(t)\le\frac{1}{\left(1-\tilde C_2\int_{t_l}^\infty F(t) dt\right)} \left(F(t_l)+\tilde C_1 \int_{t_l}^\infty F(t) dt\right).\]
Letting $T\to\infty$ yields
\[\sup_{t\in(t_l,\infty)}F(t)\le\frac{1}{\left(1-\tilde C_2\int_{t_l}^\infty F(t) dt\right)} \left(F(t_l)+\tilde C_1\int_{t_l}^\infty F(t) dt\right)\to0\quad\text{as }l\to \infty\]
which shows the claim.
\end{proof}

To prove now the convergence of the flow, we first note $u(t)$ is uniformly (in $t \in (0,\infty)$) bounded in $H^1(M,\bar g)$ by \autoref{Properties1}5. and \autoref{KonvergenzF}. We now consider a sequence $t_l\to\infty$, $l\to\infty$ and the associated sequence of functions $u_l:=u(t_l)$. This sequence is bounded in $H^1(M,\bar g)$, hence there exists a subsequence, again denoted by $(u_l)_l$, with $u_l\to u_\infty$ weakly in $H^1(M,\bar g)$ and therefore strongly  in $L^2(M,\bar g)$. Furthermore with \eqref{Lower-Upper-Boundtildealpha} we know that $\alpha_l:=\alpha(t_l)\to \alpha_\infty$ as $l\to\infty$ after passing again to a subsequence.
Moreover we claim that $\e^{\pm u_l}\to \e^{\pm u_\infty}$ (as $l\to\infty$) in $L^p(M,\bar g)$ for any $2\le p<\infty$. Indeed, using \autoref{Linfty} and the elementary estimate
\begin{equation}
  \label{EstimateExp}
|1-\e^x| \le |x| \e^{|x|} \qquad \text{for $x \in \R$,}  
\end{equation}
 we find that  
\begin{align*}
&\|\e^{u_l}-\e^{u_\infty}\|^p_{L^p(M,\bar g)}=\int_M\e^{pu_l}|1-\e^{u_\infty-u_l}|^pd\mu_{\bar g}\le \e^{p\mathcal{N}}\int_M|1-\e^{u_\infty-u_l}|^pd\mu_{\bar g}\\
&\le \e^{p\mathcal{N}}\int_M|u_\infty-u_l|^p\e^{p|u_\infty-u_l|}|d\mu_{\bar g} \le\e^{p\mathcal{N}}\e^{2p\mathcal{N}}\int_M|u_\infty-u_l|^{p-2}|u_\infty-u_l|^2d\mu_{\bar g}\\
&\le \e^{3p\mathcal{N}}(2\mathcal{N})^{p-2}\|u_\infty-u_l\|^2_{L^2(M,\bar g)}\to 0\quad\text{as }l\to \infty.
\end{align*}
Replacing $u_l$ by $-u_l$ we get also $\e^{-u_l}\to \e^{-u_\infty}$ in $L^p(M,\bar g)$ as $l\to\infty$ for any $p<\infty$.
Furthermore, we have
\begin{align*}
\|\e^{2u_l}\alpha_l-\e^{2u_\infty}\alpha_\infty\|_{L^2(M,\bar g)}&\le\|\e^{2u_l}(\alpha_l-\alpha_\infty)\|_{L^2(M,\bar g)}+\|\alpha_\infty(\e^{2u_l}-\e^{2u_\infty})\|_{L^2(M,\bar g)}\\
&\le\|\e^{2u_l}\|_{L^\infty(M,\bar g)}|\alpha_l-\alpha_\infty|A^{\frac12}+|\alpha_\infty|\|\e^{2u_l}-\e^{2u_\infty}\|_{L^2(M,\bar g)}\\
&\to 0\quad\text{for }l\to \infty.
\end{align*}
Since moreover $\e^{2u_l}\partial_tu_l\to0$ in $L^2(M,\bar g)$ as $l\to\infty$  with \autoref{Linfty} and \autoref{KonvergenzF}, the evolution equation \eqref{HG1} yields 
\begin{equation*}
\Delta_{\bar g}u_l=\e^{2u_l}\partial_tu_l+\bar K-\e^{2u_l}f+\e^{2u_l}\alpha_l \;\to\; \bar K-\e^{2u_\infty}f+\e^{2u_\infty}\alpha_\infty \qquad \text{in $L^2(M,\bar g)$.}
\end{equation*}
Since the Laplace operator $\Delta_{\bar g}$ is closed in $L^2(M,\bar g)$ with domain $H^2(M,\bar g)$, we deduce that $u_\infty$ in $H^2(M,\bar g)$ with
\begin{equation}
  \label{eq:limiting-Laplacian-equation}
\Delta_{\bar g}u_\infty = \bar K-\e^{2u_\infty}f+\e^{2u_\infty}\alpha_\infty
\end{equation}
and thus   
\[\|\Delta_{\bar g}(u_l-u_\infty)\|_{L^2(M,\bar g)}\to0\quad \text{as}\quad l\to\infty.\]
So, we even have strong convergence $u_l\to u_\infty$ in $H^2(M,\bar g)$ and uniformly, which implies that $u_\infty \in \mathcal{C}_A$ and therefore
$$
\alpha_\infty = \frac{1}{A}\left(\int_M f d\mu_{g_\infty} -\bar K\right)
$$
by integrating \eqref{eq:limiting-Laplacian-equation} over $M$.
Consequently, for the Gauss curvature $K_{g_\infty}$ of the limit metric $g_\infty=\e^{2u_\infty}\bar g$ we get from \eqref{GaussEquation} and \eqref{eq:limiting-Laplacian-equation} that
\[K_{g_\infty}=\e^{-2u_\infty}\bigl(-\Delta_{\bar g}u_\infty + \bar K\bigr)= f-\alpha_\infty=f+ \frac{1}{A}\left(\bar K-\int_Mfd\mu_{g_\infty}\right)\]
which shows the convergence of the flow.

 \subsection{\texorpdfstring{The Sign of the Constant $\alpha_\infty$}{The Sign of the Constant}} 
In this subsection we complete the proofs of \autoref{sign-changing} and \autoref{sec:stat-minim-probl-1-cor-2-theorem}. For this we  show, under certain assumptions, that the expression
\[
  \lambda =   \frac{1}{A}\left(\bar K-\int_Mfd\mu_{g_\infty}
\right)\]
is positive. The proof of \autoref{sec:stat-minim-probl-1-cor-2-theorem} is already completed by the statement of \autoref{sec:stat-minim-probl-1-cor-2}. So we can turn to \autoref{sign-changing}.

\begin{proof}[Proof of \autoref{sign-changing} (completed)]
  We have seen in \autoref{Linfty} that in the case where $u_0 \equiv \frac{1}{2}\log(A) \in \mathcal{C}_{p,A}$, the uniform $L^\infty$-bound on the global solution of the initial value problem \eqref{HG1}, \eqref{HG2} only depends on $A$ and an upper bound on $\|f\|_{L^\infty(M,\bar g)}$. In other words, if $A>0$ and $c>0$ are fixed, then there exists $\tau>0$ with the property that
\[
  \sup_{t>0}\|u(t)\|_{L^\infty(M,\bar g)}\le \tau
\]
for every $f \in C^\infty(M)$ with $\|f\|_{L^\infty(M,\bar g)} \le c$ and the corresponding solution $u$ of the initial value problem \eqref{HG1}, \eqref{HG2} with $u_0 \equiv \frac{1}{2}\log(A) \in \mathcal{C}_{p,A}$. Consequently, 
we also have $\|u_\infty\|_{L^\infty(M,\bar g)} \le \tau$ under the current assumptions on $f$, which implies that 
\begin{align*}
  \lambda&=\frac1A\left(\bar K- \int_M f \e^{2u_\infty}d\mu_{\bar g}\right)= \frac1A\left(\bar K +cA - \int_M (f+c) \e^{2u_\infty}d\mu_{\bar g}\right)\\
  &\ge c + \frac{\bar K}{A} - \|f+c\|_{L^1(M,\bar g)}\|\e^{2u_\infty}\|_{L^\infty(M,\bar g)} \ge c + \frac{\bar K}{A} - \|f+c\|_{L^1(M,\bar g)}\e^{2\tau}.     
\end{align*}
Hence, if $\|f+c\|_{L^1(M,\bar g)} <\varepsilon:=\frac{c + \frac{\bar K}{A}}{\e^{2\tau}}$,
we have $\lambda>0$. 
\end{proof}

\section{Appendix}
In this section, we collect some helpful estimates and well-posedness results for a class of linear second order parabolic equations in non-divergence form with continuous second order coefficient. Most of these results should be known to experts but seem hard to find in the required form in the literature. 

As before, let $(M,\bar g)$ be a two-dimensional, smooth, closed, connected, oriented 
Riemann manifold endowed with a smooth background metric $\bar g$. 
For a domain $\Omega \subset \R\times M$  and $p \ge 1$, we let $W^{2,1}_p(\Omega)$
denote the space of functions $u \in L^p(\Omega)$ which have weak derivatives $Du$, $D^2u$ and $\partial_tu$ in $L^p(\Omega)$. In the following, we fix $p>2$, and we recall the following embedding, see e.g.~\cite[Lemma 3.3]{LadSolUra68}.

\begin{lemma}
  \label{W12p-embedding}
 If the domain $\Omega \subset \R\times M$ is bounded, then $W^{2,1}_p(\Omega)$ is continuously embedded in $C^\alpha(\overline \Omega)$ for some $\alpha = \alpha(p)>0$ and therefore compactly embedded in $C(\overline \Omega)$. 
\end{lemma}
We consider the linear parabolic problem
\begin{equation}
  \label{eq:linear-parabolic}
\partial_tu(t,x) = a(t,x)\Delta_{\bar g} u(t,x) + c(t,x)u(t,x) +d(t,x),    
\end{equation}
with  $a,c,d \in C(\overline \Omega)$ and $d \in L^p(\Omega)$. We say that a function $u \in W^{2,1}_p(\Omega)$ is a (strong) solution of \eqref{eq:linear-parabolic} in $\Omega$ if \eqref{eq:linear-parabolic} holds almost everywhere in $\Omega$. Specifically, we consider \eqref{eq:linear-parabolic} on the cylindrical domains $\Omega_T= (0,T)\times M$ and $\widetilde \Omega_T= (-\infty,T)\times M$ in the following.

In particular, we consider strong solutions of \eqref{eq:linear-parabolic} together with the initial condition
\begin{equation}
  \label{eq:linear-parabolic-initial}
u(0,x)=u_0(x) \qquad \text{in $M$}
\end{equation}
with $u_0 \in W^{2,p}(M,\bar g)$, which is supposed to hold in the (initial) trace sense.

\begin{proposition}
\label{sec:appendix}
Let $T>0$, $a,c \in C(\overline\Omega_T)$ with $a_T:= \min \limits_{(t,x)\in\overline\Omega_T}a(t,x) >0$, let $d \in L^p(\Omega_T)$ for some $p>2$, and let $u_0 \in W^{2,p}(M,\bar g)$. 

Then the initial value problem \eqref{eq:linear-parabolic}, \eqref{eq:linear-parabolic-initial} has a unique strong solution $u \in W^{2,1}_p(\Omega_T)$. Moreover, $u$ satisfies the estimate
\begin{equation}
  \label{eq:W2-1-p-a-priori}
\|u\|_{W^{2,1}_p(\Omega_T)} \le C \Bigl( \|u_0\|_{W^{2,p}(M,\bar g)} + \|d\|_{L^p(\Omega_T)}\Bigr)
\end{equation}
with a constant $C>0$ depending only on $\|a\|_{L^\infty(\Omega_T)}$, $\|c\|_{L^\infty(\Omega_T)}$ and $a_T$. Moreover, $C$ does not increase after making $T$ smaller.\\ 
If, moreover, $a,c,d \in C^\alpha(\Omega_T)$ for some $\alpha>0$, then $u \in C(\overline \Omega_T) \cap C^{2,1}(\Omega_T)$ is a classical solution of \eqref{eq:linear-parabolic}, \eqref{eq:linear-parabolic-initial}, and we have the inequality
\begin{equation}
  \label{eq:H-1-inequality-appendix}
\|u_0\|_{H^1(M,\bar g)} \ge \limsup_{t \to 0^+}\|u(t)\|_{H^1(M,\bar g)}  
\end{equation}
\end{proposition}

\begin{proof} In the following, the letter $C$ stands for various positive constants depending only on $\|a\|_{L^\infty(\Omega_T)}$, $\|c\|_{L^\infty(\Omega_T)}$, and $a_T$, and which do not increase after making $T$ smaller.

\smallskip
{\bf Step 1:} We first assume that we are given a strong solution $u \in W^{2,1}_p(\Omega_T)$ of \eqref{eq:linear-parabolic}, \eqref{eq:linear-parabolic-initial} with $u_0\equiv0 \in W^{2,p}(M,\bar g)$. We then define $v: \widetilde\Omega_T \to \R$ by
$$
v(t,x)= \left \{
  \begin{aligned}
    &u(t,x),&& \qquad \text{for }t>0;\\
    &0,&& \qquad \text{for }t \le 0.
  \end{aligned}
\right.
$$
Then $v \in W^{2,1}_p(\widetilde \Omega_T)$ solves \eqref{eq:linear-parabolic} with $a,c,d$ replaced by suitable extensions $\tilde a,\tilde c, \in L^\infty(\widetilde \Omega_T)$, $\tilde d\in L^p(\widetilde\Omega_T)$ satisfying $\tilde a(t,x)=a(x,0)$, $\tilde c(t,x)=c(x,0)$ and $\tilde d(t,x)= 0$ for $t\le0$, $x \in M$.

Therefore, \cite[Theorem 7.22]{Lie96} gives rise to the uniform bound
\begin{equation}
  \label{eq:uniform-liebermann-v}
\|D^2 v\|_{L^p(\widetilde \Omega_T)} + \|\partial_tv\|_{L^p(\widetilde \Omega_T)} \le C \Bigl(\|\tilde d\|_{L^p(\widetilde \Omega_T)}+ \|v\|_{L^p(\widetilde \Omega_T)}\Bigr).   
\end{equation}
This translates into the estimate
\begin{equation}
  \label{eq:uniform-liebermann-u}
\|D^2 u\|_{L^p(\Omega_T)} + \|\partial_tu\|_{L^p(\Omega_T)} \le C \Bigl(\|d\|_{L^p(\Omega_T)}+ \|u\|_{L^p(\Omega_T)}\Bigr).   
\end{equation}
Moreover, setting $V(t):=\|u(t)\|_{L^p(M,\bar g)}^p$ for $t \in \R$, we have $V(0)=0$ and 
\begin{align*}
  \dot V(t)&= p\int_{M} |u(t)|^{p-2}u(t) \partial_tu(t)d\mu_{\bar g} \le pV(t)^{\frac{1}{p'}} \|\partial_tu(t)\|_{L^p(M,\bar g)}\\
  &\le 
  p\left(\frac{V(t)}{p'}+\frac{\|\partial_tu(t)\|^p_{L^p(M,\bar g)}}{p}\right)=\frac{p}{p'}V(t)+\|\partial_tu(t)\|^p_{L^p(M,\bar g)}
\end{align*}
for $t \in (0,T)$, therefore
\begin{align*}
  V(t) &= \int_0^t \dot V(s)\,ds \le \frac{p}{p'}\int_0^t V(s)\,ds + \|\partial_tu\|_{L^p(\Omega_t)}^p\\
       &\le \frac{p}{p'}\int_0^t V(s)\,ds + C \Bigl(\|d\|_{L^p(\Omega_t)}^p+ \|u\|_{L^p(\Omega_t)}^p\Bigr)\le C \left(\int_0^t V(s)\,ds + \|d\|_{L^p(\Omega_t)}^p \right).
\end{align*}
By Gronwall's inequality we get $V(t) \le C \|d\|_{L^p(\Omega_t)}^p$ and thus
\begin{equation}
  \label{eq:gronwall-consequence}
\|u(t)\|_{L^p(M,\bar g)} \le C \|d\|_{L^p(\Omega_t)}\qquad \text{for $t \in [0,T]$.}
\end{equation}
This already implies the uniqueness of strong solutions of \eqref{eq:linear-parabolic}, \eqref{eq:linear-parabolic-initial}, since the difference $u$ of two solutions $u_1,u_2 \in W^{2,1}_p(\Omega_T)$ of \eqref{eq:linear-parabolic}, \eqref{eq:linear-parabolic-initial} satisfies 
\eqref{eq:linear-parabolic}, \eqref{eq:linear-parabolic-initial} with $u_0=0$ and $d=0$.
Moreover, if $u \in W^{2,1}_p(\Omega_T)$ is a strong solution of \eqref{eq:linear-parabolic}, \eqref{eq:linear-parabolic-initial},
then the function $\hat u \in W^{2,1}_p(\Omega_T)$ given by $\hat u(t,x):= u(t,x)- u_0(x)$ safisfies
\eqref{eq:linear-parabolic}, \eqref{eq:linear-parabolic-initial} with $u_0=0$ and $d$ replaced by $\hat d$ given by
$$
\hat d(t,x)= d(t,x)+a(t,x) \Delta_{\bar g} u_0(x)+c(t,x) u_0(x).
$$
Consequently, combining \eqref{eq:uniform-liebermann-u} and \eqref{eq:gronwall-consequence}, and using an interpolation estimate for $Du$, we find that
\begin{align*}
  \|u\|_{W^{2,1}_p(\Omega_T)} &\le \|\hat u\|_{W^{2,1}_p(\Omega_T)}+\|u_0\|_{W^{2,p}(M,\bar g)}\le C\left(\|\hat d\|_{L^p(\Omega_T)}+\|\hat u\|_{L^p(\Omega_T)}\right)+\|u_0\|_{W^{2,p}(M,\bar g)}\\
  &\le C\|\hat d\|_{L^p(\Omega_T)}+\|u_0\|_{W^{2,p}(M,\bar g)}\le C\left(\|d\|_{L^p(\Omega_T)}+\|u_0\|_{W^{2,p}(M,\bar g)}\right),
\end{align*}
as claimed in \eqref{eq:W2-1-p-a-priori}.

{\bf Step 2 (Existence):} In the case where $a,c,d \in C^\alpha(\Omega_T)$ and $u_0\in C^{2+\alpha}(M)$, the existence of a classical solution $u \in C(\overline \Omega_T) \cap C^{2,1}(\Omega_T)$ of \eqref{eq:linear-parabolic}, \eqref{eq:linear-parabolic-initial} follows as in \cite[Theorem 5.14]{Lie96}. 

In the general case we consider \eqref{eq:linear-parabolic}, \eqref{eq:linear-parabolic-initial} with coefficients $a_n, c_n, d_n \in C^\alpha(\overline\Omega_T)$, $u_{0,n}\in C^{2+\alpha}(M)$, in place of $a,c,d, u_0$ with the property that $a_n \to a$, $c_n \to c$ in $L^\infty(\Omega_T)$, $d_n \to d \in L^p(\Omega_T)$ as well as $u_{0,n}\to u_0$ in $W^{2,p}$. The associated unique solutions $u_n \in C(\overline \Omega_T) \cap C^{2,1}(\Omega_T)$ are uniformly bounded in $W^{2,1}_p(\Omega_T)$ by \eqref{eq:W2-1-p-a-priori}, and therefore we have $u_n \weak u$ in $W^{2,1}_p(\Omega_T)$ after passing to a subsequence. For every $\phi \in C^\infty_c(\Omega_T)$, we then have
\begin{align*}
&\int_{\Omega_T}\Bigl(\partial_tu(t,x) - a(t,x)\Delta_{\bar g} u(t.x) - c(t,x)u(t,x)  - d(t,x)\Bigr)\phi(t,x) d\mu_{\bar g}(x)dt\\
&= \lim_{n \to \infty}     
\int_{\Omega_T}\Bigl(\partial_tu_n(t,x) - a_{n}(t,x)\Delta_{\bar g} u_n(t,x) - c_{n} (t,x)u_n(t,x)  - d_{n}(t,x)\Bigr)\phi(t,x) d\mu_{\bar g}(x)dt = 0,
\end{align*}
and from this we deduce that $\partial_tu(t,x) - a(t,x)\Delta_{\bar g} u(t,x) - c(t,x)u(t,x)  - d(t,x)= 0$ almost everywhere in $\Omega_T$, so $u$ is a strong solution of 
\eqref{eq:linear-parabolic}.\\
{\bf Step 3:} It remains to show the inequality~\eqref{eq:H-1-inequality-appendix} in the case where $a,c,d \in C^\alpha(\Omega_T)$ for some $\alpha>0$. Since $u \in C(\overline{\Omega_T}) \cap C^{2,1}(\Omega_T)$ in this case and therefore 
$$
\|u_0\|_{L^2(M,\bar g)}=  \lim_{t \to 0^+}\|u(t)\|_{L^2(M,\bar g)},  
$$
it suffices to show that 
\begin{equation}
  \label{eq:H-1-inequality-appendix-sufficient}
\|\nabla u_0\|_{L^2(M,\bar g)} \ge \limsup_{t \to 0^+}\|\nabla u(t)\|_{L^2(M,\bar g)}.
\end{equation}
If $u_0 \in C^{2+\alpha}(M)$ for some $\alpha>0$, this follows by \cite[Theorem 5.14]{Lie96} with $\lim$ in place of $\limsup$, since the function $t \mapsto u(t)$ is continuous from $[0,T) \to C^{2+\alpha}(M)$ in this case. Moreover, in this case we have, by H\"older's and Young's inequality, 
\begin{align*}
  \frac{d}{dt} \|\nabla u(t)\|_{L^2(M,\bar g)}^2 &= -\int_M \partial_tu(t) \Delta u(t)d\mu_{\bar g} \\
  &= - \int_M \Bigl(a(t)|\Delta u(t)|^2  + c(t) u(t)\Delta u(t) +d(t) \Delta u(t) \Bigr)d\mu_{\bar g}\\
                                 &\le  - a_T \|\Delta_{\bar g} u(t)\|_{L^2(M,\bar g)}^2 + \|c(t) u(t) + d(t)\|_{L^2(M,\bar g)} \|\Delta_{\bar g} u(t)\|_{L^2(M,\bar g)}\\
&\le  - a_T \|\Delta_{\bar g} u(t)\|_{L^2(M,\bar g)}^2 + a_T\|\Delta_{\bar g} u(t)\|_{L^2(M,\bar g)}^2 + \frac{1}{4a_T} \|c(t) u(t) + d(t)\|_{L^2(M,\bar g)}^2\\
&= \frac{1}{4a_T} \|c(t) u(t) + d(t)\|_{L^2(M,\bar g)}^2,
\end{align*}
and therefore
\begin{equation}
  \label{eq:intermediate-inequality-appendix}
\|\nabla u(t)\|_{L^2(M,\bar g)}^2 \le \|\nabla u(0)\|_{L^2(M,\bar g)}^2 + \frac{1}{4a_T} \int_0^t  \|c(s) u(s) + d(s)\|_{L^2(M,\bar g)}^2 \,ds \qquad \text{for $t>0$.}
\end{equation}
In the general case, we consider  \eqref{eq:linear-parabolic}, \eqref{eq:linear-parabolic-initial} with a sequence of initial conditions $u_{n,0}$ in place of $u_0$, where $u_{n,0} \to u_0$ in $H^2(M)$. The associated unique solutions $u_n \in C(\overline \Omega_T) \cap C^{2,1}(\Omega_T)$ are uniformly bounded in $W^{2,1}_p(\Omega_T)$ by \eqref{eq:W2-1-p-a-priori}, and they are also uniformly bounded in $C^{2,1}([\eps,T]\times M)$ by \cite[Theorem 5.15]{Lie96} for every $\eps \in (0,T)$. Fix $t \in (0,T)$. Passing to a subsequence, we may assume that $u_n \weak u$ in $W^{2,1}_p(\Omega_T)$, $u_n \to u$ strongly in $C^{0}(\overline{\Omega_T})$ and $u_n(t) \to u(t)$ strongly in $C^1(M)$. As in Step 2, we see, by testing with $\phi \in C^\infty_c(\Omega_T)$, that $\partial_tu(t,x) -a(t,x)\Delta_{\bar g} u(t,x) - c(t,x)u(t,x)  - d(t,x)= 0$ almost everywhere in $\Omega_T$, so $u$ is the unique strong solution of 
\eqref{eq:linear-parabolic}, \eqref{eq:linear-parabolic-initial}. Moreover, by \eqref{eq:intermediate-inequality-appendix} we have
\begin{align*}
  \|\nabla u(t)\|_{L^2(M,\bar g)}^2 &= \lim_{n \to \infty} \|\nabla u_n(t)\|_{L^2(M,\bar g)}^2 \\
  &\le \lim_{n \to \infty}\left( \|\nabla u_n(0)\|_{L^2(M)}^2 + \frac{1}{4a_T} \int_0^t  \|c(s) u_n(s) + d(s)\|_{L^2(M,\bar g)}^2 \,ds \right)\\
&= \|\nabla u(0)\|_{L^2(M,\bar g)}^2 + \frac{1}{4a_T} \int_0^t  \|c(s) u(s) + d(s)\|_{L^2(M,\bar g)}^2 \,ds. 
\end{align*}
It thus follows that
$$
\|\nabla u(t)\|_{L^2(M,\bar g)}^2- \|\nabla u(0)\|_{L^2(M,\bar g)}^2 \le \frac{1}{4a_T} \int_0^t  \|c(s) u(s) + d(s)\|_{L^2(M,\bar g)}^2 \,ds
$$
and therefore
$$
\limsup_{t \to 0} \Bigl(\|\nabla u(t)\|_{L^2(M,\bar g)}^2- \|\nabla u(0)\|_{L^2(M,\bar g)}^2\Bigr) \le \frac{1}{4a_T} \lim_{t \to 0^+} \int_0^t  \|c(s) u(s) + d(s)\|_{L^2(M,\bar g)}^2 \,ds = 0,
$$
as claimed in \eqref{eq:H-1-inequality-appendix-sufficient}.
\end{proof}

Next we prove a maximum principle for solutions of \eqref{eq:linear-parabolic},~\eqref{eq:linear-parabolic-initial}. We need the following preliminary lemma.

\begin{lemma}
  \label{prelim-1-appendix}
  Let $T>0$.  
  \begin{itemize}
  \item[(i)] For any function $u \in C^2(M)$ we have
    $$
    \int_{\{x\in M\mid u(x)>0\}}\Delta_{\bar g} u d\mu_{\bar g}\le 0.
    $$
  \item[(ii)] Let $u,\rho \in C^{1}([0,T])$ be functions with $u(0) \le 0$ and $\rho(T) \ge 0$. Then
  \begin{equation}
    \label{eq:lemma-claim-prelim}
        \int_{\{t\in[0,T]\mid u(t)>0\}}\bigl(\rho(t) \partial_tu(t) + \kappa u(t)\bigr) \,dt \ge 0 \qquad \text{with}\quad \kappa:= \sup_{s \in (0,T)}\partial_t\rho(s).
  \end{equation}
  \item[(iii)] Let $u \in C^{2,1}(\Omega_T) \cap C^{0,1}(\overline\Omega_T)$, $\rho \in C^{0,1}(\overline\Omega_T)$ be functions with $u \le 0$ on $\{0\}\times M$ and $\rho \ge 0$ on $\{T\}\times M$. Then we have
  \begin{equation}
    \label{eq:lemma-claim}
    \begin{split}
  &\int_{\{(t,x)\in [0,T]\times M\mid u(t,x)>0\}}(\rho(t,x)\partial_t u(t,x) + \kappa u(t,x) - \Delta_{\bar g} u(t,x))d\mu_{\bar g}(x)dt \ge 0 \\
  &\text{with}\quad\kappa:= \sup_{(s,x)\in (0,T)\times M}\partial_t\rho(s,x). 
  \end{split}
  \end{equation}
  \end{itemize}
  \end{lemma}
  \begin{proof}
   (i) By Lebesgue's theorem, it suffices to prove
  \begin{equation}
    \label{eq:lemma-claim-eps}
  \int_{\{x\in M\mid u(x)>\eps_n\}}\Delta_{\bar g} ud\mu_{\bar g}\le 0  
  \end{equation}
  for a sequence $\eps_n \to 0^+$. By Sard's Lemma, we may choose this sequence such that $\Omega_\eps:= \{x\in M\mid u(x)>\eps_n\}$ is an open set of class $C^1$, whereas the outer unit vector field of $\Omega_\eps$ is given by $(t,x)\mapsto- \frac{\nabla_{\bar g} u(t,x)}{|\nabla_{\bar g} u(t,x)|_{\bar g}}$. Hence \eqref{eq:lemma-claim-eps} follows from the divergence theorem.\\
   (ii) The set $\{t\in[0,T]\mid u(t)>0\}$ is a union of at most countably many open intervals $I_j$, $j \in \N$. For any such interval, partial integration gives
  $$
  \int_{I_j }\Bigl(\rho(t) \partial_tu(t)+ \partial_t\rho(t) u(t)\Bigr) \,dt = \left\{
    \begin{aligned}
    &0, && \qquad \text{if $T \not \in \overline I_j$;}\\
    &\rho(T) u(T) \ge 0, && \qquad \text{if $T \in \overline I_j$.}
    \end{aligned}
  \right.
  $$
  Consequently,
$$  
\int_{\{t\in[0,T]\mid u(t)>0\}}\rho(t) \partial_tu(t)\,dt \ge - \int_{\{t\in[0,T]\mid u(t)>0\}}\partial_t\rho(t) u(t) \,dt \ge - \int_{\{t\in[0,T]\mid u(t)>0\}}\kappa u(t) \,dt
$$
with $\kappa$ given in \eqref{eq:lemma-claim-prelim}. This shows the claim.\\
(iii) This is a direct consequence of (i), (ii) and Fubini's theorem. 
  \end{proof}

\begin{proposition}
  \label{max-principle} (Maximum principle)\\
  Let $T>0$, $a, c \in C(\overline\Omega_T)$ with $a_T:= \min \limits_{(t,x)\in\overline\Omega_T}a(t,x) >0$, let $d \in L^p(\Omega_T)$ for some $p>2$ with $d_T := \sup_{(t,x)\in\Omega_T}d(t,x)< \infty$, and let $u_0 \in W^{2,p}(M,\bar g)$. Moreover, let $u \in W^{2,1}_p(\Omega_T)$ be the unique solution of \eqref{eq:linear-parabolic}, \eqref{eq:linear-parabolic-initial}.
  \begin{itemize}
  \item[(i)] If $u_0 \le 0$ on $M$ and $d_T \le 0$, then $u \le 0$ on $\Omega_T$.
  \item[(ii)] If $c \equiv 0$ on $\Omega_T$, then  
\begin{equation}
  \label{max-principle-est}
  u(t,x) \le \|u_0^+\|_{L^\infty(M,\bar g)} + t d_T \qquad \text{for $t \in [0,T],\: x \in M$.}
\end{equation}
\end{itemize}
\end{proposition}

\begin{proof}
(i) {\bf Step 1:} We consider the special case $a \in C^{0,1}(\overline\Omega_T)$, $u_0\le 0$
and $d_T\le -\eps$ for some $\eps>0$.  We put $\rho:= \frac{1}{a} \in C^{0,1}(\overline \Omega_T)$ and $\kappa:= \sup \limits_{(s,x)\in (0,T)\times M}\partial_t\rho(s,x)$ as in~\eqref{eq:lemma-claim}. Moreover, we consider the function
$$
\breve{u}\in W^{2,1}_p(\Omega_T), \qquad \breve{u}(t,x)= \mathrm{e}^{- \breve{\kappa} t} u(t,x) 
$$
with $\breve{\kappa} = \frac{|\kappa|}{\min_{(t,x)\in \overline\Omega_T}\rho(t,x)} + \|c\|_{L^\infty(\Omega_T)}$, noting that $\breve{u}$ satisfies
\begin{equation}
  \label{eq:max-princ-mod-eq}
  \begin{split}
\rho(t,x)  \partial_t\breve{u}(t,x)& -\Delta_{\bar g} \breve{u}(t,x) + \kappa \breve{u}(t,x)\\
&=\mathrm{e}^{-\breve{\kappa}t}\Bigl(u(t,x)(\rho(t,x)c(t,x)-\rho(t,x)\breve{\kappa}+\kappa)+\rho(t,x)d(t,x)\Bigr)\\
&\le - \rho(t,x) \eps \mathrm{e}^{- \breve{\kappa} t} \qquad \text{almost everywhere in $\{(t,x)\in \Omega_T\mid\breve{u}(t,x)>0\}$.}
 \end{split}
\end{equation}

We now let $(u_n)_{n\in\N}$ be a sequence in $C^{2,1}(\Omega_T) \cap C^{0,1}(\overline\Omega_T)$ with $u_n(x,0)\le0$ and $u_n \to \breve{u}$ in $W^{2,1}_p(\Omega_T)$.
Since the functions $g_n:= 1_{\{(t,x)\in [0,T]\times M\mid u_n(t,x)>0\}}$ are bounded in $L^{p'}(\Omega_T)$, we may pass to a subsequence such that
$g_n \weak g$ in $L^{p'}(\Omega_T)$, where $g \ge 0$ and $g \equiv 1$ in $\{(t,x)\in [0,T]\times M\mid \breve{u}(t,x) >0\}$, since $u_n \to \breve{u}$ uniformly as a consequence of \autoref{W12p-embedding} and therefore $g_n \to 1$ pointwisely on $\{(t,x)\in [0,T]\times M\mid \breve{u}(t,x)>0\}$. Applying \autoref{prelim-1-appendix} (iii) to $u_n$, we find that
\begin{align*}
  0&\le \int_{\{(t,x)\in [0,T]\times M\mid u_n(t,x)>0\}} \Bigl( \rho(t,x) \partial_tu_n(t) -\Delta_{\bar g} u_n(t,x) + \kappa u_n(t,x) \Bigr) d\mu_{\bar g}(x)dt\\
  &= \int_{(0,T)\times M} g_n(t,x) \Bigl( \rho(t,x) \partial_tu_n(t,x) -\Delta_{\bar g} u_n(t,x) + \kappa u_n(t,x) \Bigr) d\mu_{\bar g}(x)dt
\end{align*} 
for all $n \in \N$ and therefore
\begin{align*}
  0 &\le \lim_{n \to \infty} \int_{(0,T)\times M} g_n(t,x) \Bigl( \rho(t,x) \partial_tu_n(t,x) -\Delta_{\bar g} u_n(t,x) + \kappa u_n(t,x) \Bigr) d\mu_{\bar g}(x)dt\\
    &= \int_{(0,T)\times M} g(t,x) \Bigl( \rho(t,x) \partial_t\breve{u}(t,x) -\Delta_{\bar g} \breve{u}(t,x) + \kappa  \breve{u}(t,x)\Bigr) d\mu_{\bar g}dt\\
    &\le - \int_{(0,T)\times M} g(t,x)  \rho(t,x) \eps \mathrm{e}^{- \breve{\kappa} t} d\mu_{\bar g}(x)dt \le - \int_{\{(t,x)\in (0,T)\times M\mid \breve{u}(t,x)>0\}}\rho(t,x) \eps \mathrm{e}^{- \breve{\kappa} t} d\mu_{\bar g}(x)dt .
\end{align*}
We thus conclude that $\{(t,x)\in (0,T)\times M\mid \breve{u}(t,x)>0\} = \{(t,x)\in (0,T)\times M\mid u(t,x) >0\} = \varnothing$ and therefore $u \le 0$ in $(0,T)\times M$.\\
{\bf Step 2:} In the special case where $a \in C^{0,1}(\overline\Omega_T)$, $u_0\le 0$
and $d_T\le 0$, we may apply Step 1 to the functions $u_\eps \in W^{2,1}_p(\Omega_T)$ defined by $u_\eps(t,x)= u(t,x)- \eps t$, which yields that $u_\eps \le 0$ for every $\eps>0$ and therefore $u \le 0$ in $\Omega_T$.\\
{\bf Step 3:} In the general case, we consider a sequence $a_n \in C^{0,1}(\overline\Omega_T)$ with $a_n \to a$ in $C(\overline \Omega_T)$, and we let $u_n$ denote the associated solutions of \eqref{eq:linear-parabolic},~\eqref{eq:linear-parabolic-initial} with $a$ replaced by $a_n$. As in the end of the proof of \autoref{sec:appendix}, we then find that, after passing to a subsequence, $u_n \weak \tilde u$ in $W^{2,1}_p(\Omega_T)$, where $\tilde u$ is a solution of \eqref{eq:linear-parabolic},~\eqref{eq:linear-parabolic-initial}. By uniqueness, we have $u= \tilde u$.
Moreover, since $u_n \le 0$ for all $n$ by Step 3, we have $u = \tilde u \le 0$, as required.\\
(ii) We consider the function $v \in W^{2,1}_p(\Omega_T)$ given by $v(t,x)= u(t,x)- \|u_0^+\|_{L^\infty(M,\bar g)}- t d_T$, which, by assumption, satisfies \eqref{eq:linear-parabolic}, \eqref{eq:linear-parabolic-initial} with $c \equiv 0$, $d- d_T$ in place of $d$ and $u_0 - \|u_0^+\|_{L^\infty(M,\bar g)}$ in place of $u_0$. Then (i) yields $v \le 0$ in $\Omega_T$, and therefore $u$ satisfies \eqref{max-principle-est}.
\end{proof}

\section{Data availability statement}
Data sharing not applicable to this article as no datasets were generated or analysed during the current study.

\printbibliography

\end{document}